\documentclass[12pt]{amsart}

\usepackage{graphicx}
\usepackage{amsfonts}
\usepackage{epsf}
\usepackage{amssymb}
\usepackage{amsmath}
\usepackage{amscd}
\usepackage{hyperref}
\usepackage{color}
\usepackage{bbm}

\newtheorem{theorem}{Theorem}[section]
\newtheorem{proposition}[theorem]{Proposition}

\newtheorem{corollary}[theorem]{Corollary}

\theoremstyle{remark}

\newtheorem{example}[theorem]{Example}
\newtheorem{remark}{Remark}[section]

\newcommand{\eps}{\epsilon}

\input xy
\xyoption{all}

\begin{document}

\title{On the periodic non-orientable 4-genus of a knot}
\begin{abstract}
We show that the equivariant and non-equivariant non-orientable 4-genus of $p$-periodic knots may differ, for any choice of $p\ge 2$. Similar results have previously been obtained for the smooth 4-genus and non-orientable 3-genus of a periodic knot. These stand in contrast to Edmods' acclaimed result by which the equivariant and non-equivariant Seifert genus of a periodic knot agree. 
\end{abstract}
%
%
\thanks{The second author was partially supported by the Simons Foundation, Award ID 524394, and by the NSF, Grant No. DMS--1906413. }
\author{Taran Grove}
\email{tarangrove@gmail.com}
\address{Department of Mathematics and Statistics, University of Nevada, Reno NV 89557.}
\author{Stanislav Jabuka}
\email{jabuka@unr.edu}
\address{Department of Mathematics and Statistics, University of Nevada, Reno NV 89557.}
\maketitle
\section{Introduction and Results}
A knot $K$ is said to have {\em period $p\ge 2$}, or to be {\em $p$-periodic}, if there exists an orientation preserving diffeomorphism $f:S^3\to S^3$ of order $p$ such that $f(K)=K$ and whose fixed point set is diffeomorphic to $S^1$. A spanning surface $\Sigma\subset S^3$ for $K$ shall be called {\em equivariant} or {\em $p$-periodic} if $f(\Sigma ) = \Sigma$.  Similarly, a smoothly and properly embedded surface $\Sigma \subset D^4$ with $\partial \Sigma =K$, shall be called {\em equivariant} or {\em $p$-periodic} provided there exists an order $p$, orientation preserving diffeomorphism $F:D^4\to D^4$ with $F|_{S^3} = f$ and with $F(\Sigma ) = \Sigma$. 

Recall these four flavors of knot genera of a knot $K$: 
\begin{itemize}
	\item[] $g_3(K)$ - the {\em Seifert} or {\em 3-genus of $K$}, 
	\item[] $g_4(K)$ - the {\em (smooth) 4-genus} or the {\em (smooth) slice genus of $K$}, 
	\item[] $\gamma_3(K)$ - the {\em non-orientable 3-genus} or {\em crosscap number of $K$}, and
	\item[] $\gamma_4(K)$ - the {\em non-orientable (smooth) 4-genus of $K$}. 
\end{itemize}
Of these, $\gamma_3(K)$ and $\gamma_4(K)$ are defined to equal the least first Betti number among all non-orientable surfaces $\Sigma $ with boundary $K$ and with $\Sigma $ embedded into $S^3$ for the case of $\gamma_3(K)$, and $\Sigma$ properly and smoothly embedded into $D^4$ for the case of $\gamma_4(K)$. 
\vskip1mm
It is an interesting question so ask how these knot genera change if, in the case of a $p$-periodic knot, one only considers $p$-periodic surfaces of the relevant type (oriented or not, embedded into $S^3$ or $D^4$). Motivated by this question, we introduce four flavors of {\em equivariant} or {\em $p$-periodic knot genera} for $p$-periodic knots, namely $g_{3,p}, g_{4,p}, \gamma_{3,p}$ and $\gamma_{4,p}$, defined in the obvious way. For example, 
$$\gamma_{4,p}(K) = \min \left\{ b_1(\Sigma) \quad \big|\quad   \begin{minipage}{90mm}  $\Sigma\subset D^4$ is a smoothly and properly embedded, equivariant, non-orientable surface with $\partial \Sigma =K$ \end{minipage} \right\}.$$
It is not difficult to see that equivariant surfaces of all four flavors, bounding a given periodic knot $K$, always exist (cf. Proposition \ref{PropositionExtindingPeriodicityToTheFourBall}), making the above definitions well posed. 
\vskip1mm
Edmonds' celebrated result \cite{Edmonds} can be reinterpreted as stating that $g_3(K) = g_{3,p}(K)$ for every $p$-periodic knot $K$.  It is known that there exist $p$-periodic knots $K$ with $g_{4,p}(K)>g_{4}(K)$, see \cite{BoyleIssa, ChaKo, DavisNaik, Naik}. The second author showed \cite{Jabuka2021} that for every odd integer $p\ge 3$ there exists a torus knot $K_p$ with $\gamma_3(K_p) = 2$ and with $\gamma_{3,p}(K_p) \ge p$, proving that difference $\gamma_{3,p}(K) - \gamma_3(K)$ can become arbitrarily large. 

The main goal of this work is to compare $\gamma_4(K)$ and $\gamma_{4,p}(K)$. Our results show that these two quantities can indeed differ, leaving the Seifert genus $g_3$ as the only genus flavor for which the equivariant and non-equivarint definitions agree. 
\begin{theorem}  \label{main}
For every integer $p\ge 2$ there exists a $p$-periodic alternating knot $K_p$ with $\gamma_{4,p}(K_p) > \gamma_4(K_p)$. One may chose $K_p$ to be slice if $p$ is even, and to be amphicheiral (and non-slice) if $p$ is odd. 

Specifically, if one let's $K_p$ denote the $p$-fold connected sum of the Figure Eight knot with itself, then
$$\gamma_4(K_p) = \left\{
\begin{array}{cl}
2 & \quad ; \quad n \text{ odd}, \cr
1 	 & \quad ; \quad n \text{ even},
\end{array}
\right.
\qquad \text{ and } \qquad 
\gamma_{4,p}(K_p) \ge  \left\{
\begin{array}{cl}
	3 & \quad ; \quad n \text{ odd}, \cr
	2 	 & \quad ; \quad n \text{ even}.
\end{array}
\right.$$  
\end{theorem}    
The techniques used to prove the preceding theorem require the knots under consideration to be alternating. Among low-crossing alternating knots there are several periodic examples, but our techniques were unable to identify any knots among them with differing equivariant and non-equivariant non-orientable 4-genus.  
\vskip3mm 
\noindent {\bf Organization } Section \ref{SectionBackground} provides the needed background for the proof of Theorem \ref{main}, with examples to illustrate definitions and theorems. Section \ref{SectionOnConnectedSumsOfFigureEightKnots} applies the concepts from Section \ref{SectionBackground} to the $p$-periodic knot $K_p$ obtained as the $p$-fold connected sum of the Figure Eight knot. The section concludes with a proof of Theorem \ref{main}.  
\vskip3mm
\noindent {\bf Acknowledgements } We wish to thank Chris Herald for helpful suggestions.

\section{Background} \label{SectionBackground}
\subsection{Goertiz matrices}  \label{SectionGoertizMatrices}
Our exposition in this section follows closely that of Gordon and Litherland \cite{GordonLitherland}. 

Consider one of the two possible checkerboard colorings of a knot diagram. An example of a $3$-periodic diagram for the knot $12a_{1019}$ and its two colorings is shown in Figure \ref{Figure12a-1019Checkerboards}. 
\begin{figure}
\includegraphics[width=16cm]{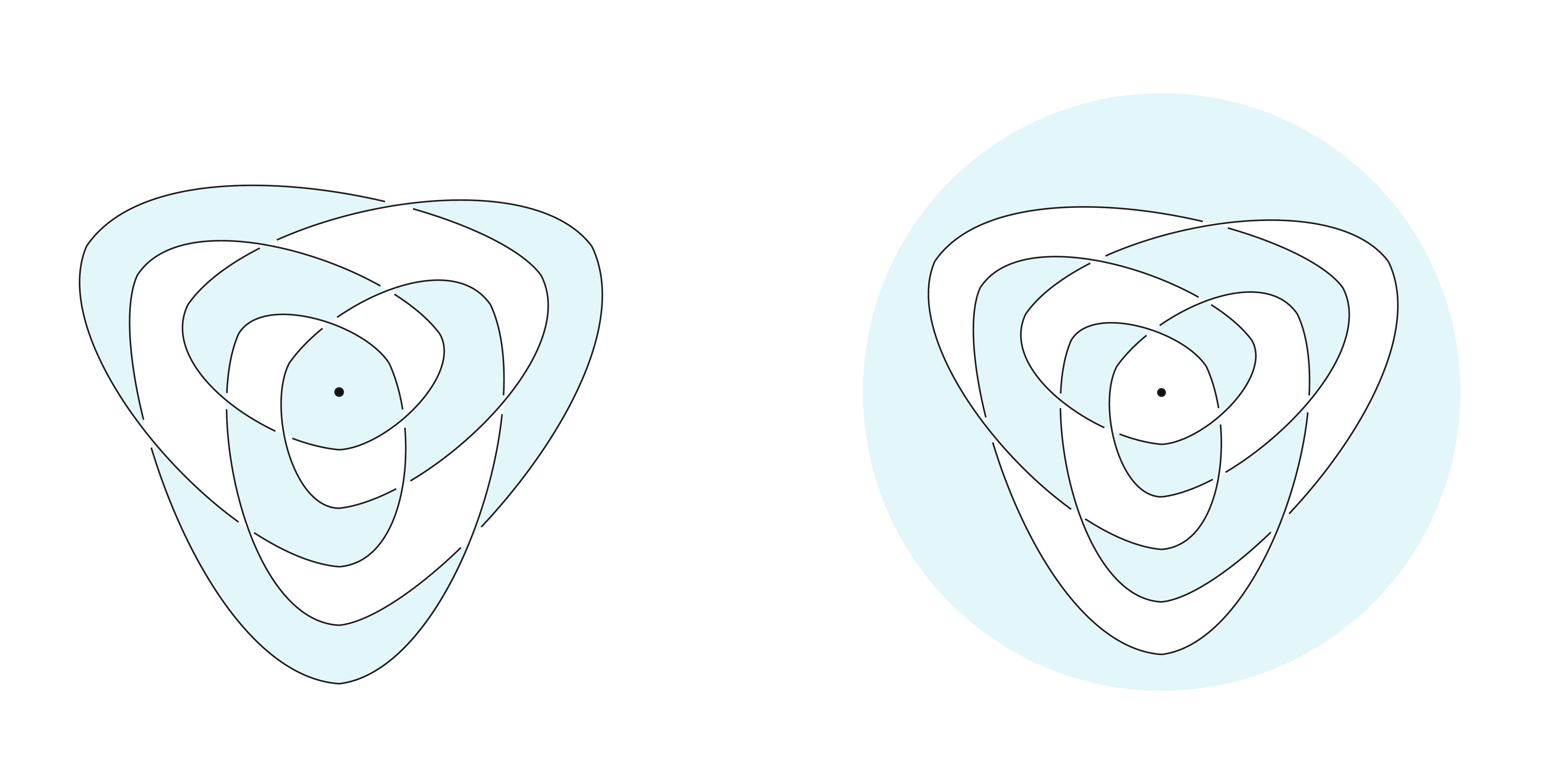}
\put(-315,-20){(a)}
\put(-85,-20){(b)}
\put(-407,105){\tiny $X_1$}
\put(-334,78){\tiny $X_2$}
\put(-345,155){\tiny $X_3$}
\put(-386,118){\tiny $X_4$}
\put(-360,88){\tiny $X_5$}
\put(-345,128){\tiny $X_6$}
\put(-295,80){\tiny $X_0$}
\put(-145,158){\tiny $X_4$}
\put(-124,44){\tiny $X_5$}
\put(-62,140){\tiny $X_6$}
\put(-140,139){\tiny $X_1$}
\put(-136,78){\tiny $X_2$}
\put(-89,120){\tiny $X_3$}
\put(-115,110){\tiny $X_0$}
\caption{The two checkerboard colorings of a $3$-periodic diagram for the knot $12a_{1019}$. The 3-periodicity is facilitated by the order 3 diffeomorphism $f:S^3\to S^3$ which is counterclockwise rotation by $2\pi/3$ about the central dot indicated in both knot diagrams. }
\label{Figure12a-1019Checkerboards} 
\end{figure}
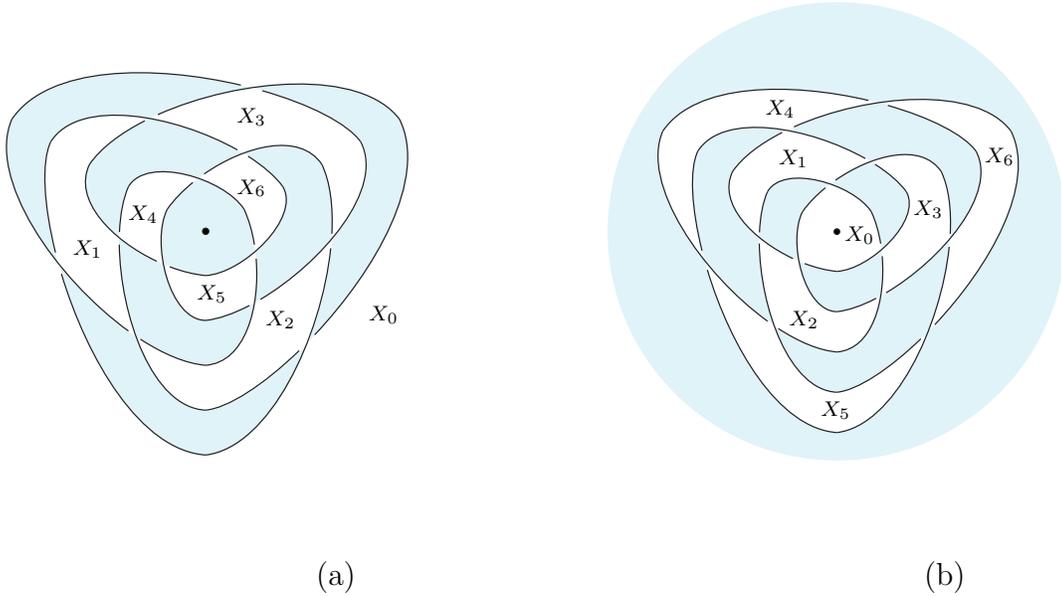
We assign weights $\eta(x) \in \{\pm 1\}$ to each double point $x$ of the knot projection according to the rule from Figure \ref{DoublePointsWeights}. For example, each of the 12 double points in the diagram in Part (a) of Figure \ref{Figure12a-1019Checkerboards} receives weight -1, while each of the 12 double points in Part (b) of the same figure receives weight 1. 

Let $X_0, \dots, X_n$ label the white regions of the checkerboard coloring (cf. Figure \ref{Figure12a-1019Checkerboards}). The the {\em pre-Goertiz matrix $PG$} is the $(n+1)\times (n+1)$ matrix with entries $[g_{i,j}]$, $i,j=0,\dots, n$, given by 
$$g_{i,j} = \left\{
\begin{array}{cl}
-\sum _{x\in X_i\cap X_j} \eta (p) & \quad ; \quad i\ne j, \cr
& \cr
-\sum _{k\ne i} g_{i,k} & \quad ; \quad i=j. 
\end{array}
\right.$$
In the above, $X_i\cap X_j$ denotes the set of double points of the knot projection that are incident to both white regions $X_i$ and $X_j$. The {\em Goertiz matrix $G=[g_{i,j}]$}, $i,j=1,\dots ,n$, is the $n\times n$ matrix obtained form $PG$ by deleting its first column and first row. One may more generally pass from $PG$ to $G$ by deleting the $i$-th column and $i$-th row in $PG$, but without loss of generality we shall always choose the first row and column, for if needed one may re-index the regions $X_0, \dots, X_n$ to begin with.   

If the knot projection is alternating, the Goeritz matrices of both checkerboard colorings are definite, one positive and one negative definite. This property characterizes alternating knots as proved by Greene \cite{Greene}. 
\begin{figure}
	\includegraphics[width=8cm]{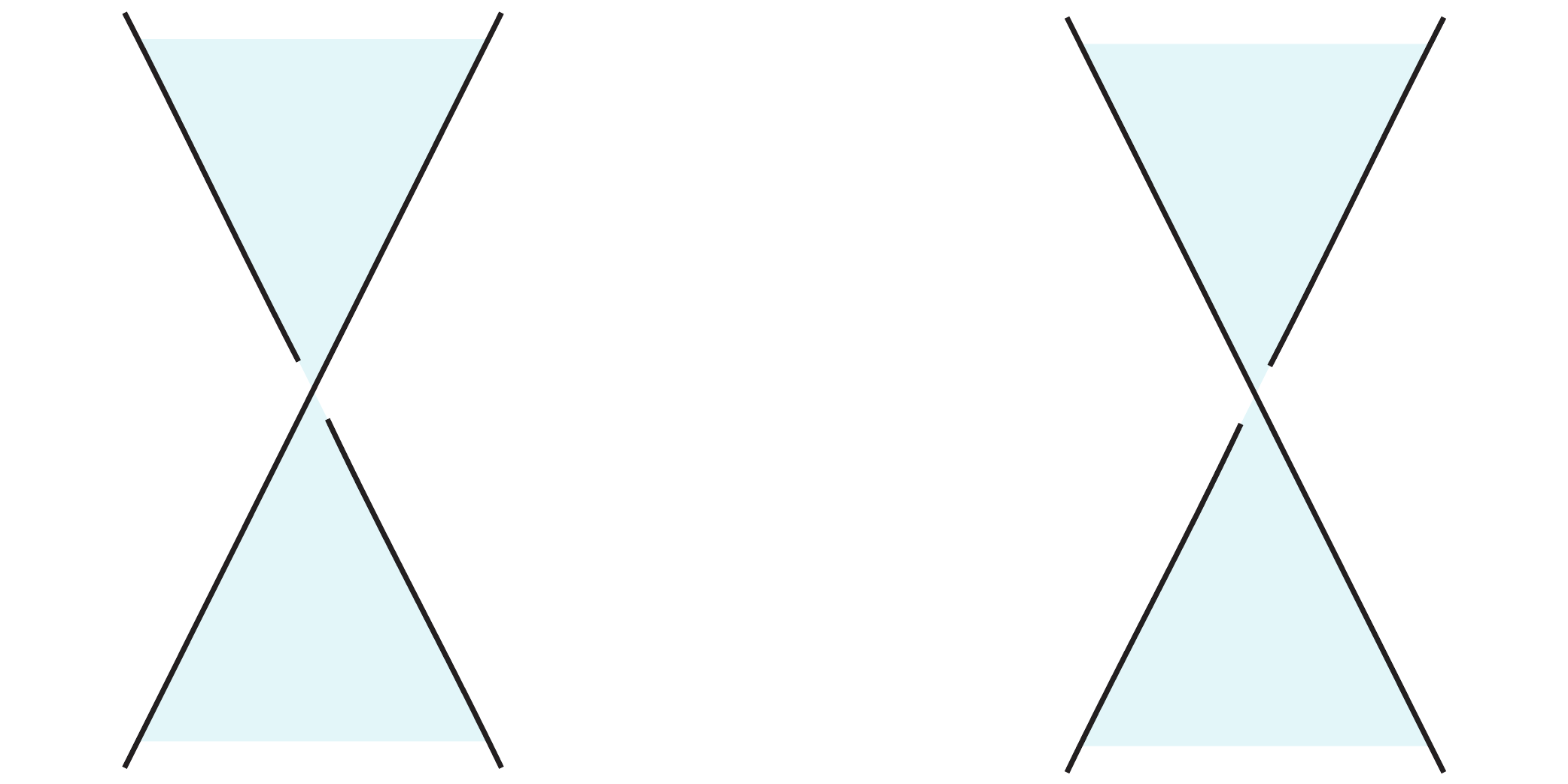}
	\put(-200,55){$x$}
	\put(-201,-20){$\eta(x) = 1$}
	\put(-35,55){$x$}
    \put(-70,-20){$\eta(x) = -1$}
	\caption{Double points in a knot projection are assigned weights $\pm 1$. }
	\label{DoublePointsWeights} 
\end{figure}
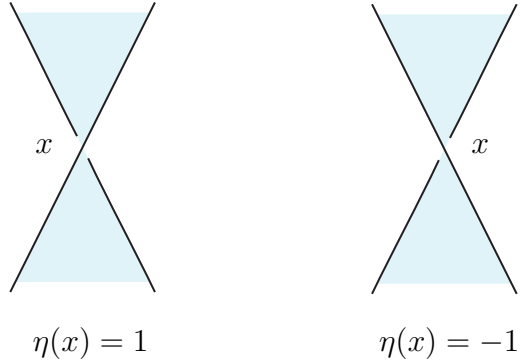
%
\begin{example}  \label{ExampleOfTheGoeritzMatricesFor12a1019}
The Goertiz matrices $G_-$ and $G_+$ associated to the checkerboard colorings in Parts (a) and (b) respectively in Figure \ref{Figure12a-1019Checkerboards}, are
	\begin{equation} \label{GoeritzMatrixFor12a-1022}
		G_- = \left[
		\begin{array}{rrrrrr}
			-4  &  1  & 1  & 1  & 0  & 0 \cr 
			1 & -4   & 1 & 0 & 1 & 0 \cr
			1 & 1 & -4   & 0 & 0 &1 \cr
			1 & 0 & 0 & -3  &  1 & 1 \cr
			0 & 1 & 0 & 1 & -3  &  1 \cr
			0 & 0 & 1 & 1 & 1 & -3 
		\end{array}
		\right],
		\qquad 
		G_+ = \left[
		\begin{array}{rrrrrr}
			4  & - 1  & -1  & -1  & 0  & 0 \cr 
			-1 & 4   & -1 & 0 & -1 & 0 \cr
			-1 & -1 & 4   & 0 & 0 &-1 \cr
			-1 & 0 & 0 & 3  &  -1 & -1 \cr
			0 & -1 & 0 & -1 & 3  &  -1 \cr
			0 & 0 & -1 & -1 & -1 & 3 
		\end{array}
		\right].
	\end{equation}
For randomly chosen knot projections, the matrices $G_-$ and $G_+$ are typically rather different, even of different ranks, but in the case above one finds that $G_+ = -G_-$. 
\end{example}

Let $\Sigma$ denote the spanning surface for $K$ resulting from the black regions of the checkerboard coloring. This surface is comprised of the black disks, connected by half-twisted bands at the double points of the projections. Note that $\Sigma$ may be orientable or not. Let $\hat \Sigma$ be obtained from $\Sigma$ by a push-in of the interior of $\Sigma$ into $D^4$ so that $\hat \Sigma$ is a smoothly and properly embedded surface in $D^4$ with $\partial \hat \Sigma=K$. Let $\hat X$ be the smooth 4-manifold obtained as the double cover of $D^4$ with branching set $\hat \Sigma$. Observe that $\partial \hat X$ is the double cover of $S^3$ branched along $K$. The following result follows from Theorems 1 and 3 in \cite{GordonLitherland}.
\begin{theorem}[Gordon--Litherland \cite{GordonLitherland}] \label{TheoremGordenLitherland}
Let $K, n, G, \Sigma, \hat \Sigma$ and $\hat X$ be as in the preceding paragraph and let  $Q_{\hat X}$ denote the intersection form on $H_2(\hat X;\mathbb Z)/Tor$. Then there is an isomorphism of lattices
$$ (H_2(\hat X;\mathbb Z)/Tor, Q_{\hat X}) \cong (\mathbb Z^n, G). $$ 
\end{theorem}
\begin{proof}[Sketch of proof]
We provide here only the part of the proof that is relevant for later constructions, and let the interested reader find complete details in \cite{GordonLitherland}. 

The branched covering space $\hat X$ can be constructed by gluing together two copies $D_1$ and $D_2$ of the manifold obtained by cutting open $D^4$ along the trace of the isotopy from $\Sigma$ to $\hat \Sigma$. Each of $D_1$ and $D_2$ is itself diffeomorphic to $D^4$ and the exposed portion after cutting along the trace of the isotopy, corresponds to a tubular neighborhood $N=N(\Sigma)$ of $\Sigma$ in $S^3$. Let $\iota :N\to N$ be the involution given by reflection in the fiber (viewing here $N$ as an $I$-bundle over $\Sigma$), then 
$$\hat X = D_1 \cup _{x\sim \iota(x)} D_2, \qquad x\in N.$$
The Mayer-Vietoris sequence for this decomposition of $\hat X$ shows that the connecting homomorphism $H_2(\hat X;\mathbb Z) \to H_1(N;\mathbb Z)$ is an isomorphism. Indeed, its inverse $S:H_1(N;\mathbb Z) \to H_2(\hat X; \mathbb Z)$ is given by $S([\alpha]) = [S(\alpha)]$ for where a 1-cycle $\alpha$ in $N$, $S(\alpha)$ is geometrically constructed as:
$$S(\alpha ) = (\text{Cone of $\alpha$ in } D_1) - (\text{Cone of $\iota(\alpha)$ in } D_2).$$
Let $j:\Sigma\to N$ be the inclusion map, then the composition $S\circ j_\ast : H_1(\Sigma ; \mathbb Z) \to H_1(\hat X; \mathbb Z)$ is an isomorphism. 

The connection between the white regions $X_0, \dots, X_N$ of the knot diagram and $H_1(\Sigma;\mathbb Z)$ is as follows: Let $S^2\subset S^3$ be the sphere on which the knot diagram lives. Each white region $X_i$, $i=0,\dots, n$ inherits an orientation from $S^2$, and we let $\alpha_i=\partial X_i$ be the oriented boundary of $X_i$, viewed as a 1-cycle on $\Sigma$. It is easy to see that $[\alpha _0]+[\alpha _1]+\dots +[\alpha _n]=0\in H_1(\Sigma ; \mathbb Z)$ and that $\{[\alpha _0], \dots, [\alpha _n]\} - \{[\alpha _i]\}$ is a basis for $H_1(\Sigma ; \mathbb Z)$ for any choice of $i\in \{0, \dots, n\}$. It remains to show that for the basis $\{[\alpha _1], \dots, [\alpha _n]\}$ of $H_1(\Sigma ; \mathbb Z)$ one obtains 
$$G(X_i, X_j) = Q_{\hat X}\left( S\circ j_\ast ([\alpha _i])\, , \, S\circ j_\ast ([\alpha _j] )\right), $$
which is omitted here but is part of the proof of Theorem 1 in \cite{GordonLitherland}.   
\end{proof}
\subsection{Definiteness of certain branched covers}  \label{SectionOnBranchedCovers}
Consider a smoothly and properly embedded M\"obius band $M$ in $D^4$, and let $\tilde X$ be the twofold cover of $D^4$ with branching set $M$. Then $b_2(\tilde X) = b_1(M) = 1$ and thus $\tilde X$ is either positive or negative definite. For some knots the type of definiteness of $\tilde X$ can be determined as in \cite{GilmerLivingston}: $\tilde X$ is positive (negative) definite if 
$$ \sigma (K) + 4\cdot \text{Arf}\,(K) \equiv 2 \pmod 8 \qquad (\sigma (K) + 4\cdot \text{Arf}\,(K) \equiv -2 \pmod 8).$$
If $\sigma (K) + 4\cdot \text{Arf}\,(K) \equiv 0 \pmod 8$ then no such determination can be made, while if $\sigma (K) + 4\cdot \text{Arf}\,(K) \equiv 4 \pmod 8$, $K$ cannot bound a M\"obius band in $D^4$ (see again \cite{GilmerLivingston}). 
\vskip1mm
If $K$ bounds a smoothly and properly embedded punctured Klein Bottle $B$ in $D^4$, let $\tilde X$ again denote the twofold cover of $D^4$ this time branched along $B$. Then $b_2(\tilde X) =b_1(B) = 2$. Compared to the case of $K$ bounding a M\"obius band, in the present case it becomes more difficult to tell when $\tilde X$ is a definite manifold. Nevertheless, we have the following result from \cite{JabukaKelly}. 
\begin{theorem}[Jabuka-Kelly \cite{JabukaKelly}] \label{JabukaKelly}
If a knot $K$ bounds a puncture Klein Bottle $B$ smoothly and properly embedded in $D^4$ and if $\sigma (K) +4\cdot \text{Arf}(K)\equiv 4 \pmod 8$, then the twofold cover $\tilde X$ of $D^4$ branched along $B$ is definite (either positive or negative). 
\end{theorem}
\subsection{An application of Donaldson's Theorem} \label{SectionOnDonaldsonsTheorem}
Pick an alternating diagram for the alternating knot $K$, and let $G_{\pm}$ be the positive/negative definite Goeritz matrices associated to the two checkerboard colorings of said diagram. Let $\Sigma _{\pm}$ be the spanning surfaces of $K$ formed by the black regions in either checkerboard coloring, and let $\hat \Sigma_\pm $ be obtained from $\Sigma_\pm$ by pushing their interior into $D^4$. Let $\hat X_\pm $ be the twofold cover of $D^4$ branched along $\hat \Sigma_\pm$ and note that $\hat X_+$ is a positive definite, while $\hat X_-$ is negative definite and both are compact and smooth 4-manifolds. 

Suppose that $K$ bounds a non-orientable surface $S$, smoothly and properly embedded into the 4-ball $D^4$, and let $\tilde X$ be the twofold cover of $D^4$ branched along $S$.  Let 
\begin{equation} \label{EquationConstructionOfTheClosedFourManifold}
	Z_\pm =\hat X_\pm \cup _{\partial \hat X = \partial \tilde X}(- \tilde X).
\end{equation}
Then $Z_\pm $ is smooth, closed and $Z_+$ is positive definite if $\tilde X$ is negative definite, and $Z_-$ is negative definite if $\tilde X$ is positive definite. Note that $Z_\pm$ can alternatively be obtained as the twofold cover of $S^4$ branched along the non-orientable closed surface $\hat \Sigma\cup _{\partial \hat \Sigma = \partial S} S$.

An application of Donaldson's foundational theorem on intersection forms of definite, smooth, oriented, closed 4-manifolds \cite{Donaldson}, when applied to the manifolds $Z_\pm$, yields the next theorems, one for the case of $S$ being a M\"obius band, and one for the case of $S$ being a puncture Klein Bottle. Implicit in the statements of the theorems is fact that the inclusion induced map   
$$H_2(\hat X_\pm;\mathbb Z)\oplus H_2(-\tilde X;\mathbb Z) \to H_2(Z_\pm;\mathbb Z)$$  
is injective, with cokernel a finite group of order that divides $\det K$. 
\begin{theorem} \label{DonaldsonsEmbeddingTheoremForMobiusBands}
Pick an alternating projection of an alternating knot $K$, let $G_\pm$ be the associated positive/negative definite Goeritz matrices, and let $\hat X_\pm$ be the positive/negative definite 4-manifold obtained as the twofold cover of $D^4$ with branching set $\hat \Sigma_\pm$, the push-in into $D^4$ of the black surfaces $\Sigma_\pm$ from the checkerboard colorings. 

If $K$ bounds a M\"obius band $M$ in $D^4$ and if $\tilde X$ is the twofold branched cover of $D^4$ along $M$, then there are embeddings of lattices of equal rank: 
$$
\begin{array}{ll}
(\mathbb Z^{n_+}\oplus \mathbb Z, G_+\oplus Q_{\tilde X})  \hookrightarrow 	(H_2(Z_+;\mathbb Z) , Q_{Z_+})\cong (\mathbb Z^{n_++1},  \text{Id}), & \quad  \text{ if $\tilde X$ is negative definite}, \cr
(\mathbb Z^{n_-}\oplus \mathbb Z, G_-\oplus Q_{\tilde X})  \hookrightarrow 	(H_2(Z_-;\mathbb Z) , Q_{Z_-})\cong (\mathbb Z^{n_-+1},  -\text{Id}),  & \quad  \text{ if $\tilde X$ is positive definite}. 
\end{array}  
$$
Here $(\mathbb Z^{n+1},\pm$Id$)$ are the standard positive/negative definite lattice, and the integer $n_\pm$ is the number of white regions minus 1, in the checkerboard colorings used to find $G_\pm$. 
\end{theorem}
\begin{corollary} \label{CorollaryAboutCorankOneEmbeddingOfLattices}
In the setting of Theorem \ref{DonaldsonsEmbeddingTheoremForMobiusBands}, there exist embeddings of lattices 
$$(\mathbb Z^{n_\pm}, G_\pm) \hookrightarrow (\mathbb Z^{n_\pm+1}, \pm \text{Id}),$$
with equal signs chosen on both sides of the embedding.   
\end{corollary}
Relying on Theorem \ref{JabukaKelly}, we get the following analogue of Theorem \ref{DonaldsonsEmbeddingTheoremForMobiusBands} for the case of a punctured Klein Bottle replacing the M\"obius band. 
\begin{theorem} \label{DonaldsonsEmbeddingTheoremForKleinBottles}
Pick an alternating projection of an alternating knot $K$, let $G_\pm$ be the associated positive/negative definite Goeritz matrices, and let $\hat X_\pm$ be the positive/negative definite 4-manifold obtained as the twofold cover of $D^4$ with branching set $\hat \Sigma_\pm$, the push-in into $D^4$ of the black surfaces $\Sigma_\pm$ from the checkerboard colorings. If $K$ bounds a puncutre Klein Bottle $B$ in $D^4$ and if $\tilde X$ is the twofold branched cover of $D^4$ along $B$, then there are embeddings of lattices of equal rank: 
$$
\begin{array}{ll}
(\mathbb Z^{n_+}\oplus \mathbb Z^2, G_+\oplus Q_{\tilde X})  \hookrightarrow 	(H_2(Z_+;\mathbb Z) , Q_{Z_+})\cong (\mathbb Z^{n_++2},  \text{Id}), & \quad  \text{ if $\tilde X$ is negative definite}, \cr
(\mathbb Z^{n_-}\oplus \mathbb Z^2, G_-\oplus Q_{\tilde X})  \hookrightarrow 	(H_2(Z_-;\mathbb Z) , Q_{Z_-})\cong (\mathbb Z^{n_-+2},  -\text{Id}),  & \quad  \text{ if $\tilde X$ is positive definite}. 
\end{array}  
$$
Here $(\mathbb Z^{n+2},\pm$Id$)$ are the standard positive/negative definite lattice and the integer $n_\pm$ is the number of white regions minus 1, in the checkerboard colorings for $G_\pm$. 
\end{theorem}
The $\mathbb Z^2$-summands in the statement of Theorem \ref{DonaldsonsEmbeddingTheoremForKleinBottles} represent the free part of $H_2(\tilde X;\mathbb Z)$.  Combining Theorem \ref{DonaldsonsEmbeddingTheoremForKleinBottles} with Theorem \ref{JabukaKelly}, we obtain this useful consequence:
\begin{corollary} \label{CorollaryAboutCorankTwoEmbeddingOfLattices}
Pick an alternating projection of an alternating knot $K$ with 
$$\sigma(K) + 4\cdot \text{Arf}(K)\equiv 4\pmod 8.$$
Let $G_\pm$ be the associated positive/negative definite Goeritz matrices. If the lattice embeddings 
$$
(\mathbb Z^{n_+}, G_+)  \hookrightarrow 	(\mathbb Z^{n_++2},  \text{Id}),  \quad \text{ and } \quad (\mathbb Z^{n_-}, G_-)  \hookrightarrow 	(\mathbb Z^{n_-+2},  -\text{Id}) 
$$
do not exist, then $K$ cannot bound a smoothly and properly embedded punctured Klein Bottle in $D^4$. 
\end{corollary}
\begin{remark}
The main utility of  Theorems \ref{DonaldsonsEmbeddingTheoremForMobiusBands} and \ref{DonaldsonsEmbeddingTheoremForKleinBottles} is to impose a restriction on the lattice embeddings. Indeed the  'lattice' restriction coming from said theorems  is so strong that there are usually only a handful of such embeddings if they exists at all, often times they are unique.  We illustrate this with the next example, which also highlights our techniques for finding all such embeddings. 
\end{remark}
\begin{example} \label{ExampleEmbeddingOfTheKnot12a1019}
Consider again the knot $K=12a_{1019}$ and its Goeritz matrix $G_-$ as in  Example \ref{ExampleOfTheGoeritzMatricesFor12a1019}. Since $K$ is slice, it bounds a smooth M\"obius band $M$ in $D^4$. If $\tilde X$ (the twofold cover of $D^4$ branched along $M$) were positive definite, then according to Corollary \ref{CorollaryAboutCorankOneEmbeddingOfLattices} there would exist an embedding of lattices $\varphi : (\mathbb Z^6, G_-)\to (\mathbb Z^7, -\text{Id})$, which we now proceed to determine. Note that we don't know if $\tilde X$ is positive definite, this is a working hypothesis for now. 
	
If we write $G_-=[g_{i,j}]_{i,j=1,\dots, 6}$ and recall that $\mathbb Z^6$ is the free Abelian group generated by $\{X_1, \dots, X_6\}$ (cf. Part (a)  of Figure \ref{Figure12a-1019Checkerboards}), then $G_-(X_i, X_j) = g_{i,j}$. For brevity of notation, we shall simply write $X_i \cdot X_j$ to mean $G_-(X_i, X_j)$. Similarly let $\{Y_1, \dots, Y_7\}$ be the standard basis for $(\mathbb Z^7,-\text{Id})$, and write $Y_i\cdot Y_j$ to mean -Id$(Y_i, Y_j)$. Thus $Y_i \cdot Y_j = -\delta _{i,j}$. With this notation in place, we seek to find all possible $\varphi$ (up to post-composition by an isomorphism of $\mathbb Z^7$) such that $\varphi(X_i) \cdot \varphi(X_j) = X_i \cdot X_j$ for $i,j =1, \dots, 6$.

If we write $\varphi(X_i) = \sum _{k=1}^7 a_{i,k}Y_k$, $a_{i,k}\in \mathbb Z$, then all the coefficients $a_{i,k}$ are constraint to come from the set $\{-1, 0, 1\}$. For $i=4, 5, 6$ this is easy to see as $X_i\cdot X_i = -3 =-(a_{i,1}^2+\dots +a_{i,7}^2)$. For $i=1,2,3$, and since $X_i\cdot X_i=-4$, it is conceivable that one could get $\varphi(X_i) = 2Y_k$ for some index $k$. However, then $X_i\cdot X_j = \varphi(X_i)\cdot \varphi(X_j) = 2Y_k\cdot \varphi(X_j)$ would be an even integer for every index $j\in \{1,\dots , 6\}$.  By inspection one can verify that for each $i=1,2, 3$ there exists a $j$ such that $X_i\cdot X_j=1$. It follows that each $\varphi(X_i)$ is a linear combination of $\{Y_1,\dots, Y_7\}$ with coefficients from $\{-1, 0, 1\}$. This observation makes solving the equations below significantly easier. 

Turning now to the determination of $\varphi$, note that since $X_4\cdot X_4=-3$, then up to isomorphism we must have,  
$$\varphi(X_4) = Y_1+Y_2+Y_3.$$
Since $X_5\cdot X_4=1$, then $\varphi(X_5)$ shares 1 or 3 basis elements with $\varphi(X_4)$. If it shared 3 basis elements, it would follow that $X_5\cdot X_i \equiv X_4\cdot X_i \pmod 2$ for every $i=1,\dots, 6$, a congruence that fails for $i=1$. Thus $\varphi(X_5)$ and $\varphi(X_4)$ share exactly 1 basis element, and without loss of generality we can assume that the shared element is $Y_1$, leading to 
$$\varphi(X_5) = -Y_1+Y_4+Y_5.$$
By the same argument, since $X_6\cdot X_4=1=X_6\cdot X_5$, $\varphi(X_6)$ must share exactly one basis element with $\varphi(X_4)$ and one basis element with $\varphi(X_5)$. That shared basis element cannot be $Y_1$ since in that case we would get $\varphi(X_6) \cdot \varphi (X_4) = - \varphi(X_6)\cdot \varphi(X_5)$, a contradiction since $\varphi(X_6) \cdot \varphi (X_4) = 1= \varphi(X_6)\cdot \varphi(X_5)$. Without loss of generality we can assume that $\varphi(X_6)$ shares $Y_2$ with $\varphi(X_4)$ and that $\varphi(X_6)$ shares $Y_4$ with $\varphi(Y_5)$, leading to 
$$\varphi(X_6) = -Y_2-Y_4+Y_6.$$
Write $\varphi(X_1) = \sum _{k=1}^7 a_k Y_k$ with $a_k\in \{-1,0,1\}$ to be determined from the intersection pairings: $X_1\cdot X_1=-4$, $X_1\cdot X_4=1$, $X_1\cdot X_5=0$ and $X_1\cdot X_6=0$. These lead to 
$$a_1^2+\dots +a_7^2=4,\quad  -a_1-a_2-a_3 = 1, \qquad a_1-a_4-a_5=0, \quad a_2+a_4-a_6=0. $$
This is easy to solve and leads to these 6 solutions (up to isomorphism) :
$$
\begin{array}{clccl}
(i) & \varphi(X_1) = Y_1-Y_2-Y_3+Y_4, & \qquad \qquad & (iv)& \varphi(X_1)  = -Y_2 +Y_4-Y_5-Y_7, \cr
(ii) & \varphi(X_1) = -Y_1 +Y_2-Y_3-Y_4, & \qquad \qquad & (v)& \varphi(X_1)  = -Y_3+Y_4-Y_5+Y_6, \cr
(iii) & \varphi(X_1) = -Y_1 -Y_4-Y_6-Y_7, & \qquad \qquad & (vi)& \varphi(X_1)  = -Y_3-Y_4+Y_5-Y_6.
\end{array}
$$
Similarly we write $\varphi(X_2) = \sum _{k=1}^7 b_k Y_k$ with $b_k\in \{-1,0,1\}$, subject to the relations $X_2\cdot X_2=-4$, $X_2\cdot X_4=0$, $X_2\cdot X_5=1$, $X_2\cdot X_6=0$, $X_2\cdot X_1=1$, the latter equation of course has to be evaluated for each of the cases (i)--(iv):
$$
\begin{array}{clccl}
	(i) & \varphi(X_2) = Y_2-Y_3-Y_4-Y_7, & \qquad \qquad & (iv)(b)& \varphi(X_2)  = Y_2-Y_3-Y_4-Y_7,  \cr
	(ii) & \text{No solution for } \varphi(X_2), & \qquad \qquad &  (v)(a) & \varphi(X_2)  = -Y_2+Y_3-Y_5-Y_6, \cr
	(iii)(a) & \varphi(X_2) = Y_2 -Y_3-Y_5+Y_6, & \qquad \qquad & (v)(b)& \varphi(X_2)  = Y_1-Y_2-Y_6+Y_7,  \cr
	(iii)(b) & \varphi(X_2) = Y_1 -Y_2-Y_6+Y_7, & \qquad \qquad &  (vi)(a)& \varphi(X_2)  = -Y_2+Y_3-Y_5-Y_6,\cr
	(iv)(a) & \varphi(X_2)  = -Y_1 +Y_2-Y_4-Y_5, & \qquad \qquad &(vi)(b)& \varphi(X_2)  = Y_2-Y_3-Y_5+Y_6.
\end{array}
$$
The 6 possibilities for $\varphi(X_1)$ have increased to give us 9 possibilities for $\varphi(X_2)$. Writing $\varphi(X_3) = \sum _{k=1}^7c_k Y_k $ with $c_k \in \{-1,0,1\}$, we solve the corresponding system of equations and find:
$$
\begin{array}{clccl}
	(i) & \text{No solution for } \varphi(X_3), & \qquad \qquad & (iv)(b)& \text{No solution for } \varphi(X_3),  \cr
	(ii) & \text{No solution for } \varphi(X_2), & \qquad \qquad &  (v)(a) & \varphi(X_3)= Y_1-Y_3+Y_5-Y_6, \cr
	(iii)(a) & \text{No solution for } \varphi(X_3), & \qquad \qquad & (v)(b)& \text{No solution for } \varphi(X_3),  \cr
	(iii)(b) & \text{No solution for } \varphi(X_3), & \qquad \qquad &  (vi)(a)& \text{No solution for } \varphi(X_3),\cr
	(iv)(a) & \text{No solution for } \varphi(X_3), & \qquad \qquad &(vi)(b)& \varphi(X_3)  = -Y_1+Y_3-Y_5-Y_6.
\end{array}
$$
Thus only Cases (v)(a) and (vi)(b) are viable, leading to the two embeddings of lattices $\varphi_k:(\mathbb Z^6, G_-) \to (\mathbb Z^7, -\text{Id})$, $k=1,2$, given by 
\begin{equation} \label{EquationEmbeddingFor12a1019}
\begin{array}{clccl}
	\varphi_1(X_1) & = -Y_3-Y_4+Y_5-Y_6, & \qquad \qquad & \varphi_1(X_4) & = Y_1+Y_2+Y_3,  \cr
	\varphi_1(X_2) & = Y_2-Y_3-Y_5+Y_6, & \qquad \qquad &  \varphi_1(X_5) &  = -Y_1+Y_4+Y_5, \cr
	\varphi_1(X_3) & = -Y_1+Y_3-Y_5-Y_6, & \qquad \qquad & \varphi_1(X_6) & = -Y_2-Y_4+Y_6. \cr
	&&&& \cr
	\varphi_2(X_1) & = -Y_3+Y_4-Y_5+Y_6, & \qquad \qquad & \varphi_2(X_4) & = Y_1+Y_2+Y_3,  \cr
	\varphi_2(X_2) & = -Y_2+Y_3-Y_5-Y_6, & \qquad \qquad &  \varphi_2(X_5) &  = -Y_1+Y_4+Y_5, \cr
	\varphi_2(X_3) & = Y_1-Y_3+Y_5-Y_6, & \qquad \qquad & \varphi_2(X_6) & = -Y_2-Y_4+Y_6.
\end{array}
\end{equation}
It is not hard to see that these two embeddings are not isomorphic. In light of Corollary \ref{CorollaryAboutCorankOneEmbeddingOfLattices} and the fact that $12a_{1019}$ is slice, it is unsurprising that $\varphi$ exists. The utility of the present example is to demonstrate that the 'lattice condition' of the embedding from Corollary \ref{CorollaryAboutCorankOneEmbeddingOfLattices}, severely restricts the possible embeddings. 

Since $G_+ = -G_-$ (cf. Example \ref{ExampleOfTheGoeritzMatricesFor12a1019}), it follows that the only embeddings of lattices $\psi:(\mathbb Z^6, G_+) \to (\mathbb Z^7, \text{Id})$ are again given by $\psi= \varphi_k$, $k=1,2$ as in \eqref{EquationEmbeddingFor12a1019}. 
\end{example}
\subsection{The equivariant setting} \label{SectionEquivariantSetting}
In this section we re-examine Theorems \ref{DonaldsonsEmbeddingTheoremForMobiusBands} and \ref{DonaldsonsEmbeddingTheoremForKleinBottles} and Corollaries \ref{CorollaryAboutCorankOneEmbeddingOfLattices} and \ref{CorollaryAboutCorankTwoEmbeddingOfLattices} in the equivariant setting, when the knot under consideration is periodic. 

Thus, let $K$ be a $p$-periodic knot and $f:S^3\to S^3$ be an order $p$, orientation preserving diffeomorphism of $S^3$ realizing the symmetry of $K$. A {\em $p$-periodic spanning surface} (orientable or not) {\em for $K$ in $D^4$} is a pair $(\Sigma , F)$ consisting of a spanning surface $\Sigma$ for $K$, smoothly and properly embedded in $D^4$, and an orientation preserving diffeomorphism $F:D^4\to D^4$ of order $p$ such that $F|_{S^3} = f$ and $F(\Sigma ) = \Sigma$.  The next proposition shows that every periodic knot possesses  periodic spanning surfaces in $D^4$. 

\begin{proposition} \label{PropositionExtindingPeriodicityToTheFourBall}
Let $f:S^3\to S^3$ be an orientation preserving diffeomorphism of order $p$, with non-empty fixed-point set. 
\begin{itemize}
\item[(i)] There exists an orientation preserving diffeomorphism $F:D^4\to D^4$ of order $p$, with $F|_{S^3} = f$.
\item[(ii)] A $p$-periodic knot $K$ in $S^3$ admits both orientable and non-orientable $p$-periodic spanning surfaces $(\Sigma, F)$ in $D^4$.  
\end{itemize} 
\end{proposition}
\begin{proof}
(i)  Viewing $\mathbb R^4$ as $\mathbb C^2$, let $\rho_q:\mathbb C^2 \to \mathbb C^2$ be the rotation about the axis $\{0\}\times \mathbb C$ by an angle $2\pi q/p$, that is let $\rho_q(z_1, z_2) = (e^{2\pi q i /p} z_1, z_2)$. Here $q$ is an integer in $\{1,\dots, p-1\}$ relatively prime to $p$. Note that $\rho_q$ is a smooth map, inducing smooth maps (of the same name)  $\rho_q: S^3 \to S^3$ and $\rho_q:D^4\to D^4$ by restriction (where, as is customary, we view $S^3$ and $D^4$ as the unit sphere and unit ball in $\mathbb C^2$ respectively, centered at the origin).  
	
By the resolution of the Smith Conjecture \cite{BassMorgan}, any finite order, orientation preserving diffeomorphism $f:S^3\to S^3$ with nonempty fixed-point set, is up to conjugation equal to a rotation $\rho_q$ . Specifically, there exists a diffeomorphism $\phi:S^3\to S^3$ such that $f= \phi^{-1}\circ \rho_q \circ \phi$ for some choice of $q$ as above. 


Let $\Phi:D^4\to D^4$ be any diffeomorphism with $\Phi|_{S^3}=\phi$ \cite{Hatcher}, and let $F=\Phi\circ \rho_q\circ \Phi^{-1}$. Clearly $F$ is smooth and orientation preserving (since $\rho_q$ preserves orientation), and $F^k =  \Phi \circ (\rho_q)^k \circ \Phi^{-1}$, showing that the order of $F$ is $p$. It is obvious that $F|_{S^3} = f$. 
\vskip3mm
\noindent (ii) Let $S\subset S^3$ be a spanning surface for $K$, orientable or not, with $f(S) = S$. The existence of equivariant orientable spanning surfaces follows from \cite{Edmonds}. A non-orientable equivariant spanning surface can be obtained from an orientable equivariant surface $S$ by attaching a small half-twisted band $b$ to the boundary of $S$ (performing a Reidemeister move of Type I on the knot $K$), followed by the attachment of the bands $f^i(b)$, $i=2, \dots, p-1$ to $\partial S$.      

Let $f=\varphi \circ \rho_q\circ \varphi^{-1}$ and $F=\Phi\circ\rho_q\circ \Phi^{-1}$ be as in Part (i) of the proof. Let $S'=\phi^{-1}(S)$, so that $\rho_q(S') = S'$. Let $\Sigma '\subset D^4$ be obtained from $S'$ by pushing the interior of $S'$ radially into $D^4$. Note that then $\rho_q(\Sigma') = \Sigma'$. Let $\Sigma = \Phi(\Sigma')$, then $\partial \Sigma  = K$ and $F(\Sigma ) = \Sigma$. Lastly, $\Sigma$ is orientable or not, depending on whether $S$ is orientable or not.  
\end{proof}

Next we turn to the manifolds $Z=Z_\pm$ from \eqref{EquationConstructionOfTheClosedFourManifold} in the equivariant setting. Let  $f:S^3\to S^3$ be an orientation preserving, order $p>1$ diffeomorphism with Fix$(f)\ne \emptyset$, and let $K$ be a knot in $S^3$ with $f(K)=K$. Let $\Sigma _1$ be the black surface in a checkerboard coloring of an $f$-equivariant diagram of $K$, and let $\hat \Sigma_1\subset D^4$ and $F_1:D^4\to D^4$ be constructed from these as in the proof of Proposition \ref{PropositionExtindingPeriodicityToTheFourBall}. Suppose that $K$ possesses a non-orientable $p$-periodic spanning surface $(\Sigma_2, F_2)$ in $D^4$, then  
$$\partial \hat \Sigma _1 = K = \partial \Sigma _2, \quad F_1(\hat \Sigma _1) = \hat \Sigma _1, \quad F_2(\Sigma _2) = \Sigma _2, \quad \text{ and } \quad F_1\big|_{S^3} = f = F_2\big|_{S^3}.$$ 
We glue the surfaces $\hat \Sigma _1$ and $\Sigma _2$ along their common boundary $K$ to obtain the embedded closed surface $\Sigma\subset S^4$, that is we let $(S^4, \Sigma ) = (D^4, \hat \Sigma _1)\cup _{(S^3, K)}(-D^4, \Sigma _2)$. Then $Z$ is the twofold cylcic cover of $S^4$ branched along $\Sigma$. We let $F:S^4\to S^4$ be the diffeomorphism that restricts to $F_1$ and $F_2$ on the two copies of $D^4$, note that then $\Sigma$ is preserved by $F$. We wish to show that $F$ lifts to a diffeomorphism $\tilde F:Z \to Z$ that covers $F$. Our proof of this shall rely on Theorem 9.2 on page 64 from \cite{Bredon}.
\begin{theorem}[Bredon \cite{Bredon}, Page 64] \label{BredonsResult}
	Let $G$ be a Lie group (not necessarily connected) acting on a connected and locally arcwise connected space $Y$, and let $\pi:\tilde Y\to Y$ be a covering space of $Y$. Let $\tilde y_0\in \tilde Y$ project to $y_0\in Y$ and suppose that $G$ leaves $y_0$ stationary. Then there exists a (unique) $G$-action on $\tilde Y$ leaving $\tilde y_0$ stationary and covering the given action on $Y$ if and only if the subgroup $\pi_\#(\pi_1(\tilde Y,\tilde y_0))$ is invariant under the action of $G$ on $\pi_1(Y, y_0)$. 
\end{theorem}

With this in mind we have the following lifting theorem of a cylic group action to the twofold cyclic branched cover. 
\begin{theorem} \label{TheoremAboutLiftingOfDiffeomorphism}
Let $(S^4, \Sigma) = (D^4, \hat \Sigma _1)\cup _{(S^3, K)}(-D^4, \Sigma_2)$ and $F:S^4\to S^4$ be as described above. Let $f=F\big|_{S^3}:S^3\to S^3$ be the induced diffeomorphism on $S^3$ and assume that the fixed-point set Fix$(f)$ is non-empty. Let $\pi:Z \to S^4 $ be the twofold cyclic cover of $S^4$ branched along $\Sigma$.  Then there exists a diffeomorphism $\tilde  F:Z\to Z$ covering $F$. 
\end{theorem}
\begin{proof}
Let $N=N(\Sigma)$ denote a closed equivariant tubular neighborhood of $\Sigma$ (the existence of which follows from Theorem VI.2.2 in \cite{Bredon}), note that $N$ and $\partial N$ are orientable manifolds. Let $Y=S^4-\text{Int}(N)$ so that $\partial Y = \partial N$, pick a basepoint $y_0\in Y\cap ( \text{Fix}(f)-\Sigma)$ and let 
$$H=\text{Ker}(\pi_1(Y,y_0)\to H_1(Y;\mathbb Z)).$$ 
Alexander duality and the fact the $\Sigma $ is non-orientable, imply that $H_1(Y;\mathbb Z) \cong \mathbb Z_2$ (the cyclic group of order 2), showing $H$ to be a normal subgroup of index 2. Let $\pi':\tilde Y \to Y$ be the (twofold) covering space corresponding to $H$. 

On the other hand, let $\rho :N\to \Sigma$ be the map endowing $N$ with the structure of a disk-bundle over $\Sigma$. By abuse of notation, we let $\rho$ also denote the restriction-induced map $\rho:\partial N \to \Sigma$, making $\partial N$ into a circle bundle over $\Sigma$. Let $\pi'':\tilde N \to N$ be the twofold cover of $N$ branched along $\Sigma$, constructed by taking fiber-wise twofold covers of disks (the fibers of $N$) branched along their centers. Then $\partial \tilde Y \cong \partial \tilde N$ and we can arrange that under this diffeomorphism $\pi'\big|_{\partial \tilde Y} = \pi''\big|_{\partial \tilde N}$.   

The twofold cyclic covering space $\pi:Z \to S^4$ of $S^4$ branched along $\Sigma$ can then be obtained as the fiber sum  $Z=\tilde Y \cup _{\partial \tilde Y=\partial \tilde N} \tilde N$, by letting $\pi=\pi'$ on $\tilde Y$ and letting $\pi=\pi''$ on $\tilde N$.    
\vskip1mm

The Hurewicz Theorem readily implies that $(F\big|_Y)_\#(H) \subset H$ so that Theorem \ref{BredonsResult} applies to the action of  $G=\{\text{Id}_{S^4}, F\big|_Y, (F\big|_Y)^2, \dots, (F\big|_Y)^{p-1}\}\cong \mathbb Z_p$ on $Y$. Thus, given a choice of $\tilde y_0 \in (\pi')^{-1}(y_0)$ there exists a unique lift of the action of $G$ on $Y$ to an action of $G$ on $\tilde Y$, and we let $\tilde F_1:\tilde Y\to \tilde Y$ be the diffeomorphism of $\tilde Y$ covering $F:Y\to Y$. 

Since $N$ is equivariant with respect to $F$ and since $\pi\circ \tilde F = F\circ \pi$, we see that $\tilde F_1$ preserves the fibers of the circle bundle $\tilde \rho :\partial \tilde Y = \partial \tilde N \to \Sigma$. As such, $\tilde F_1:\partial \tilde N\to \partial \tilde N$ can be extended to a smooth map $\tilde F_2: \tilde N \to \tilde N$, yielding a smooth map $\tilde F:Z\to Z$ that covers the diffeomorphism $F:S^4\to S^4$. 

Our proof is rather similar to, and inspried by, the proof of Proposition 12 in \cite{BoyleIssa}.
\end{proof}
Next we combine Theorems \ref{DonaldsonsEmbeddingTheoremForMobiusBands}, \ref{DonaldsonsEmbeddingTheoremForKleinBottles} and \ref{TheoremAboutLiftingOfDiffeomorphism} to obtain concrete obstructions to a periodic knot bounding an equivariant M\"obius band or an equivariant punctured Klein Bottle in $D^4$. Before stating the results, we need a few preliminaries. 

Let again $f:S^3\to S^3$ be an orientation preserving, order $p>1$ diffeomorphism with Fix$(f) \ne \emptyset$, and $K$ be an alternating knot with $f(K)=K$. Let $G_\pm$ be the positive/negative definite Goertiz matrices of $K$ associated to an $f$-equivarant knot diagram for $K$ with white regions $X_0, X_1, \dots X_n$. Let $\mathbb Z^n$ be the free Abelian group generated by $\{X_1, \dots, X_n\}$ and let by abuse of notation, $f:\mathbb Z^n\to \mathbb Z^n$ be the isomorphism induced by the action of $f$ on the set $\{X_1, \dots, X_n\}$. For instance, in the setting of Figure  \ref{Figure12a-1019Checkerboards}, the action of $f$ (given by counterclockwise rotation by $2\pi/3$) on the set $\{X_1, \dots, X_6\}$ is given by  
\begin{equation} \label{EquationActionOn12a1019}
X_1\stackrel{f}{\longrightarrow} X_2\stackrel{f}{\longrightarrow} X_3\stackrel{f}{\longrightarrow} X_1\quad \text{ and } \quad  X_4\stackrel{f}{\longrightarrow} X_5\stackrel{f}{\longrightarrow} X_6\stackrel{f}{\longrightarrow} X_4.	
\end{equation}
Let $\Sigma_1$ be the black surface of this equivariant knot diagram and let $\hat \Sigma_1$ be a push-in of the interior of $\Sigma _1$ into $D^4$. 
Let $F_1:D^4\to D^4$ be an extension of $f$ as constructed in Proposition \ref{PropositionExtindingPeriodicityToTheFourBall} and let $\hat F_1:\hat X_1\to \hat X_1$ be the lift of $F_1$ to the cyclic 2-fold cover $\hat X_1$ of $D^4$ branched along $\hat \Sigma _1$ (constructed as in Theorem \ref{TheoremAboutLiftingOfDiffeomorphism}). Then the action of $f$ on $\mathbb Z^n$, the free Abelian group generated by $\{X_1, \dots, X_n\}$, and the induced action $(\hat F_1)_\ast:H_2(\hat X_1;\mathbb Z) \to H_2(\hat X_1;\mathbb Z)$ are compatible under the isomorphism $\theta :\mathbb Z^n \to H_2(\hat X_1; \mathbb Z)/Tor$ from Theorem \ref{TheoremGordenLitherland}, in that they lead to the commutative diagram
\begin{equation} \label{EquationCommutativeDiagramforfandF1}
\xymatrix{
	\mathbb Z^n  \ar[d]_f  \ar[r]^-\theta & H_2(\hat X_1;\mathbb Z)/Tor \ar[d]^{(\hat F_1)_\ast} \\
	\mathbb Z^n      \ar[r]^-\theta     &H_2(\hat X_1;\mathbb Z)/Tor
}
\end{equation}
Verifying this claim is a matter of working through the definitions of $\theta$ and $\tilde F_1$ ensuring their compatibility, something that is carried out in full detail in Proposition 13 of \cite{BoyleIssa}. With this understood, we have:

\begin{corollary} \label{CorollaryEquivariantDiagramForMobiusBand}
Let $f:S^3\to S^3$, $\{X_1, \dots, X_n\}$ and $G_\pm$ be as above. If $K$ bounds an equivariant M\"obius band $M$ in $D^4$, there there exists a commutative diagram 
$$
\xymatrix{
	(\mathbb Z^n, G_+)  \ar[d]_f  \ar[r]^\varphi & (\mathbb Z^{n+1}, \text{Id}) \ar[d]_{\cong}^{\tilde F} \\
	(\mathbb Z^n, G_+) \ar[r]^\varphi          &(\mathbb Z^{n+1}, \text{Id})
}	\qquad \text{ or } \qquad 
\xymatrix{
	(\mathbb Z^n, G_-)  \ar[d]_f  \ar[r]^\varphi & (\mathbb Z^{n+1}, -\text{Id}) \ar[d]_{\cong}^{\tilde F} \\
	(\mathbb Z^n, G_-) \ar[r]^\varphi          &(\mathbb Z^{n+1}, -\text{Id})
}
$$
according to whether the twofold cyclic cover $\tilde X$ of $D^4$ branched along $M$, is positive or negative definite. The horizontal map $\varphi$ is an embedding of lattices, while $\tilde F$ is an isomorphism. 
\end{corollary}
\begin{corollary} \label{CorollaryEquivariantDiagramForKleinBottle}
Let $f:S^3\to S^3$, $\{X_1, \dots, X_n\}$ and $G_\pm$ be as above, and assume that $\sigma (K)+4\cdot\text{Arf}(K)\equiv 4 \pmod 8$. If there do not exist commuative diagrams
$$
\xymatrix{
	(\mathbb Z^n, G_+)  \ar[d]_f  \ar[r]^\varphi & (\mathbb Z^{n+2}, \text{Id}) \ar[d]_{\cong}^{\tilde F} \\
	(\mathbb Z^n, G_+) \ar[r]^\varphi          &(\mathbb Z^{n+2}, \text{Id})
}	\qquad \text{ and } \qquad 
\xymatrix{
	(\mathbb Z^n, G_-)  \ar[d]_f  \ar[r]^\varphi & (\mathbb Z^{n+2}, -\text{Id}) \ar[d]_{\cong}^{\tilde F} \\
	(\mathbb Z^n, G_-) \ar[r]^\varphi          &(\mathbb Z^{n+2}, -\text{Id})
}
$$
then $\gamma_{4,p}(K)\ge 3$. In the diagrams, the horizontal maps $\varphi$ are embeddings of lattices, while $\tilde F$ is an isomorphism. 
\end{corollary}
Corollaries \ref{CorollaryEquivariantDiagramForMobiusBand} and \ref{CorollaryEquivariantDiagramForKleinBottle}  are a direct consequence of Corollaries \ref{CorollaryAboutCorankOneEmbeddingOfLattices} and \ref{CorollaryAboutCorankTwoEmbeddingOfLattices}, and of Theorem \ref{TheoremAboutLiftingOfDiffeomorphism} and the commutative diagram \eqref{EquationCommutativeDiagramforfandF1}. 
\subsection{Summary} \label{SummarySection}
Let $K$ be a $p$-periodic knot whose periodicity is realized by a diffeomorphism $f:S^3\to S^3$ with Fix$(f) \ne \emptyset$. 

To obstruct $K$ from bounding an equivariant M\"obius band $M$ in $D^4$, and thus showing that $\gamma_{4,p}(K)\ge 2$, we first find all possible lattice embeddings 
$$\varphi:(\mathbb Z^n, G_\pm)\to (\mathbb Z^{n+1}, \pm \text{Id})$$ 
as in Corollary \ref{CorollaryAboutCorankOneEmbeddingOfLattices} (see also Example \ref{ExampleEmbeddingOfTheKnot12a1019}).  With a concrete map $f:\mathbb Z^n \to \mathbb Z^n$ read off from the periodic knot diagram, we ask whether an isomorphism $\tilde F:\mathbb Z^{n+1}\to \mathbb Z^{n+1}$ as in Corollary \ref{CorollaryEquivariantDiagramForMobiusBand} exists. If not, $K$ cannot bound an equivariant M\"obius band in $D^4$. 

\vskip1mm
Similalry, to obstruct a knot $K$ with $\sigma (K) + 4\cdot \text{Arf}(K)\equiv 4\pmod 8$ from bounding a puncture Klein Bottle $B$ in $D^4$, and thus showing that $\gamma_{4,p}(K)\ge 3$, one  first finds all possible lattice embeddings 
$$\varphi:(\mathbb Z^n, G_\pm)\to (\mathbb Z^{n+2}, \pm \text{Id}),$$
compare to Corollary \ref{CorollaryAboutCorankTwoEmbeddingOfLattices}. With a concrete map $f:\mathbb Z^n \to \mathbb Z^n$ read off from the periodic knot diagram, one asks whether an isomorphism $\tilde F:\mathbb Z^{n+2}\to \mathbb Z^{n+2}$ as in Corollary \ref{CorollaryEquivariantDiagramForKleinBottle} exists. If not, $K$ cannot bound an equivariant punctured Klein bottle in $D^4$. 
\begin{example}
Consider once more the $3$-periodic knot $K=12a_{1019}$ as in Examples \ref{ExampleOfTheGoeritzMatricesFor12a1019} and \ref{ExampleEmbeddingOfTheKnot12a1019}. The Goeritz matrices $G_\pm$ were obtained in the former, while the only possible lattice embeddings $\varphi_{1,2}:(\mathbb Z^6, G_\pm) \to (\mathbb Z^7, \pm \text{Id})$ were determined in the latter example. The map $f:\mathbb Z^6\to \mathbb Z^6$ is given in \eqref{EquationActionOn12a1019}. 

Isomorphisms $\tilde F_i:\mathbb Z^7\to\mathbb Z^7$, $i=1,2$ with $\tilde F_i \circ \varphi_i = \varphi _i \circ f$ (as in Corollary \ref{CorollaryEquivariantDiagramForMobiusBand}) indeed exist, and are given by 
$$\tilde F_1=\left[
\begin{array}{cccccc}
	0 & -1 & 0 & 0 & 0 & 0 \\
	0 & 0 & 0 & -1 & 0 & 0 \\
	0 & 0 & 0 & 0 & 0 & 1 \\
	1 & 0 & 0 & 0 & 0 & 0 \\
	0 & 0 & 1 & 0 & 0 & 0 \\
	0 & 0 & 0 & 0 & 1 & 0 \\
\end{array}
\right] \quad \text{ and } \quad 
\tilde F_2=\left[
\begin{array}{cccccc}
	0 & -1 & 0 & 0 & 0 & 0 \\
	0 & 0 & 0 & -1 & 0 & 0 \\
	0 & 0 & 0 & 0 & 0 & 1 \\
	1 & 0 & 0 & 0 & 0 & 0 \\
	0 & 0 & 1 & 0 & 0 & 0 \\
	0 & 0 & 0 & 0 & 1 & 0 \\
\end{array}
\right].
$$ 
Thus in this example we cannot conclude that $\gamma_{4,3}(12a_{1019})\ge 2$. Indeed we were able to find the required isomorphisms $\tilde F$ from Corollary \ref{CorollaryEquivariantDiagramForMobiusBand} for all periodic low-crossing knots that we examined, showing the limitations of these techniques. 
\end{example}

\section{Connected sums of the Figure Eight knot} \label{SectionOnConnectedSumsOfFigureEightKnots}
For $n\in \mathbb N$ let $K_n=\#^n (4_1)$ denote the $n$-fold connected sum of the Figure Eight knot $4_1$. Since $4_1\#4_1$ is a slice knot, it follows that so is $K_n$ when $n$ is even, while if $n$ is odd then $K_n$ is concordant to $4_1$. Recall that $\gamma_4(4_1) = 2$ (the first proof of this uses intersection forms of 4-manifolds, Viro \cite{Viro}) and since $\gamma_4$ is a concordance invariant, we obtain 
\begin{equation} \label{Gamma4OfKn}
\gamma_4(K_n) = \left\{ 
\begin{array}{cl}
	2 & \quad ; \quad n \text{ is odd}, \cr
	1 & \quad ; \quad n \text{ is even}.
\end{array}
\right.
\end{equation}
Since $\sigma(K_n)=0$ for all $n\in \mathbb N$ and since Arf$(K_n)$ is 1 for $n$ odd and 0 for $n$ even, we also obtain 
\begin{equation} \label{SigmaPlusArfOfKn}
\sigma(K_n) +4\cdot \text{Arf}(K_n) \equiv  \left\{ 
\begin{array}{cl}
	4 \pmod 8 & \quad ; \quad n \text{ is odd}, \cr
	0 \pmod 8 & \quad ; \quad n \text{ is even}.
\end{array}
\right.
\end{equation}
Consider the checkerboard coloring of the $n$-periodic projection of $K_n$ as in Figure \ref{CheckerboardDiagramForFourfoldConnectedSumOfFigureEightKnot} (which depicts the case of $n=4$). The $n$-periodicity of $K_n$ is realized by the counterclockwise rotation $f$ (also shown in Figure \ref{CheckerboardDiagramForFourfoldConnectedSumOfFigureEightKnot}) about the central dot by an angle of $2\pi/n$.  
\begin{figure}
\includegraphics[width=11cm]{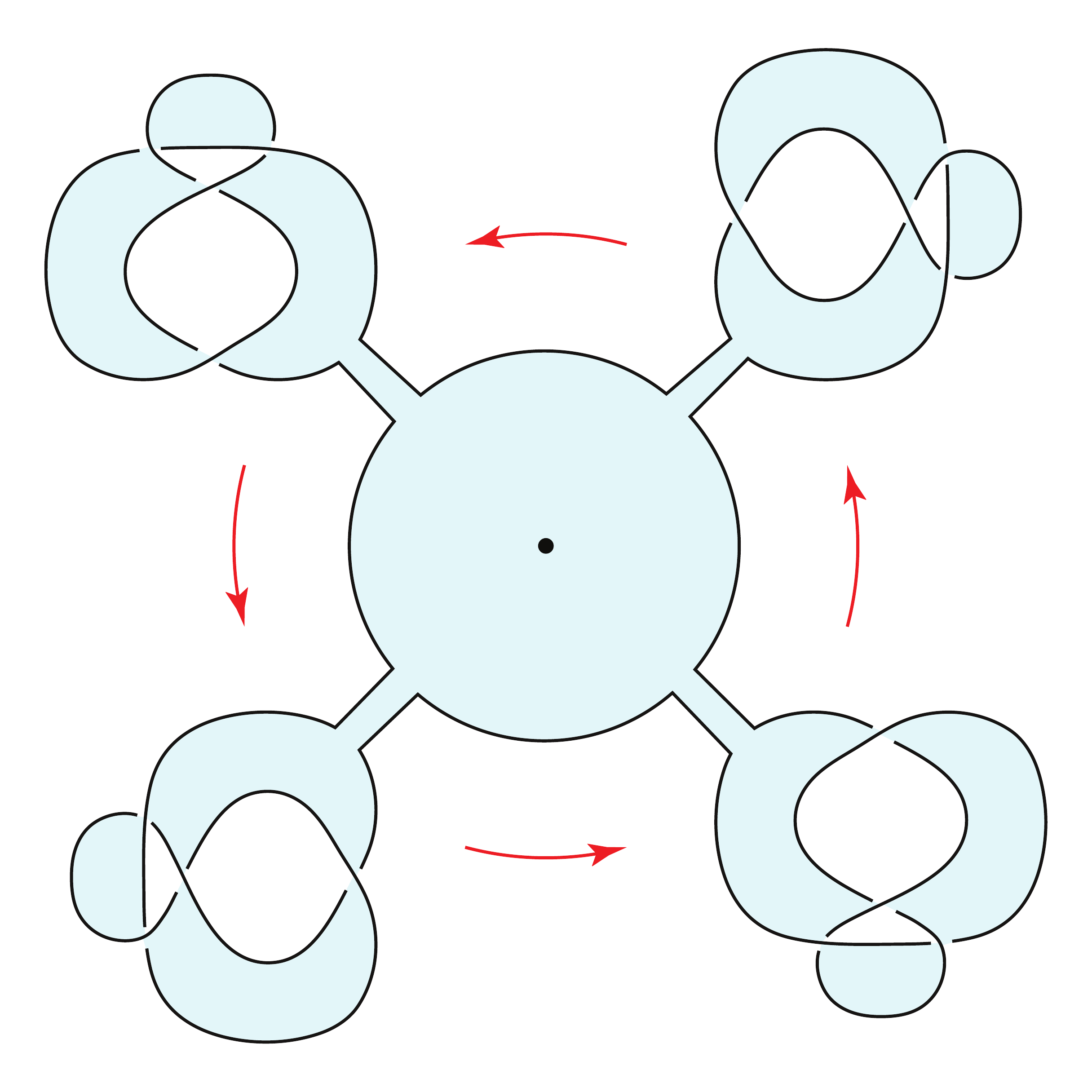}
\put(-300,158){\tiny $X_0$}
\put(-259,263){\tiny $X_1$}
\put(-259,234){\tiny $X_2$}
\put(-272,60){\tiny $X_3$}
\put(-242,60){\tiny $X_4$}
\put(-66,45){\tiny $X_5$}
\put(-66,75){\tiny $X_6$}
\put(-53,249){\tiny $X_7$}
\put(-83,249){\tiny $X_8$}
\put(-258,160){$f$}
\put(-60,154){$f$}
\put(-154,255){$f$}
\put(-162,52){$f$}
\caption{Checkerboard diagram for $K_n$, shown here for $n=4$. }
\label{CheckerboardDiagramForFourfoldConnectedSumOfFigureEightKnot} 
\end{figure}
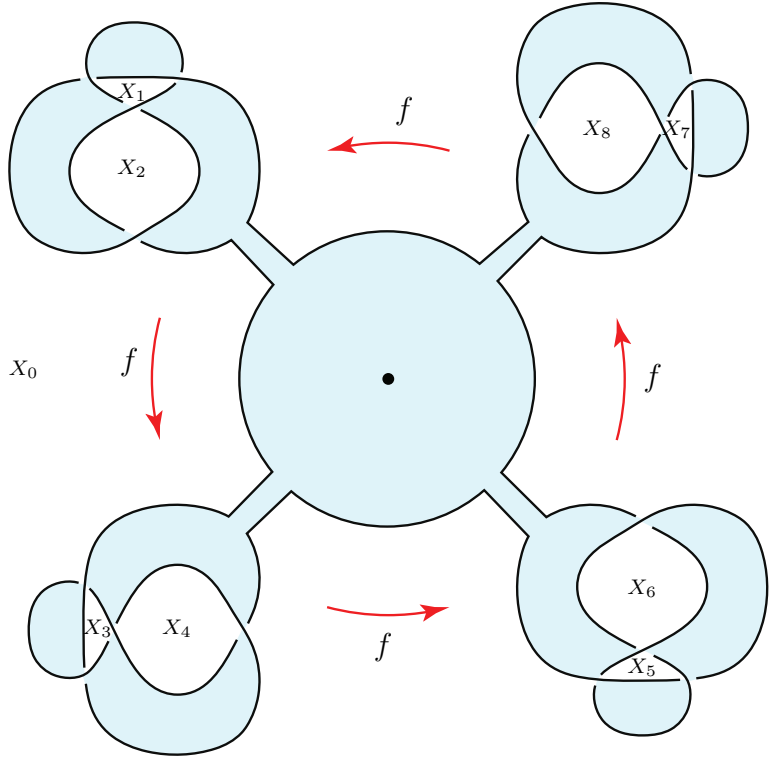
Let $X_0, X_1, \dots, X_{2n}$ label the white regions in the checkerboard diagram, with $X_0$ labeling the unbounded region. Then the Goeritz matrix $G_n$ obtained from this data (after discarding the region labeled by $X_0$) is given by 
$$G_n=\left[
\begin{array}{rr}
-3 & 1 \cr 
1 & -2 	
\end{array}
\right] \oplus \left[
\begin{array}{rr}
	-3 & 1 \cr 
	1 & -2 	
\end{array}
\right] \oplus \dots \oplus 
\left[
\begin{array}{rr}
	-3 & 1 \cr 
	1 & -2 	
\end{array}
\right]. 
$$
The direct sum above refers to block-diagonal matrix summation, and there are $n$ summands in all, one for each copy of $4_1$ in $K_n$. Note that the rank of $G_n$ is $2n$ and that the $k$-th $2\times 2$ matrix summand in $G_n$ is generated by the regions $\{X_{2k-1}, X_{2k}\}$. We view $G_n$ as a symmetric, non-degenerate intersection form on $\mathbb Z^{2n}$, the free Abelian group generated by $\{X_1, \dots, X_{2n}\}$. We shall write $X_i \cdot X_j$ to mean $G_n(X_i, X_j)$.  
\begin{remark} \label{RemarkAboutPositiveDefiniteMatrixForKn}
The Goeritz matrix $G_n$ is the negative definite Goeritz matrix $G_-$ of the alternating knot $K_n$. The positive definite Goeritz matrix $G_+$, derived from the opposite checkerboard coloring of that from Figure \ref{CheckerboardDiagramForFourfoldConnectedSumOfFigureEightKnot},  is easily seen to equal $-G_-$, something we leave this as an exercise for the interested reader. 
\end{remark}
Given \eqref{Gamma4OfKn} and \eqref{SigmaPlusArfOfKn}, we are interested in enumerating all possible  lattice embeddings $(\mathbb Z^{2n}, G_n)\hookrightarrow (\mathbb Z^{2n+2}, -\text{Id})$ when $n$ is odd, and all possible lattice embeddings  $(\mathbb Z^{2n}, G_n)\hookrightarrow (\mathbb Z^{2n+1}, -\text{Id})$ when $n$ is even. These are considered separately in the next two subsections. 
\subsection{Embeddings in the case of $n$ even}
In this section we aim to find all lattice embeddings $\varphi :(\mathbb Z^{2n}, G_n) \to (\mathbb Z^{2n+1}, -\text{Id})$ up to isomorphism. We retain the notation from above by which $\mathbb Z^{2n}$ is the free Abelian group generated by $\{X_1,\dots, X_{2n}\}$, and we similarly let $\mathbb Z^{2n+1}$ be the free Abelian group generated by $\{Y_1, \dots, Y_{2n+1}\}$. We write $X_i \cdot X_j$ and $Y_k\cdot Y_\ell$ as shorthand for $G_n(X_i, X_j)$ and -Id$(Y_k, Y_\ell)$. With this notation understood, the task at hand is to find all monomorphisms $\varphi :\mathbb Z^{2n} \to \mathbb Z^{2n+1}$ with $\varphi(X_i)\cdot \varphi(X_j) = X_i \cdot X_j$, $i, j = 1, \dots, 2n$.   

When $i\in \{1,\dots, 2n\}$ is an even index then $X_i\cdot X_i=-2$, and so $\varphi (X_i) = \eps_{i_1} Y_{i_1}+ \eps_{i_2} Y_{i_2}$ for some pair of distinct indicies $i_1, i_2 \in \{1, \dots, 2n+1\}$ and for a pair of signs $\eps_{i_1}, \eps _{i_2}\in \{\pm 1\}$. Moreover, for $i\ne j$ if follows that $\{i_1, i_2\}\cap \{j_1, j_2\} = \emptyset$. Indeed, if the intersection $\{i_1, i_2\}\cap \{j_1, j_2\}$ were non-empty and contained just one index, e.g. if we had $i_1=j_1$ but $i_2\ne j_2$, then we would find $\varphi(X_i) \cdot \varphi(X_j)= \eps_{i_1}\cdot \eps_{j_1} \ne 0$, contradicting the fact that $X_{i}\cdot X_j =0$. If we had $\{i_1,i_2\} = \{j_1, j_2\}$ then we would obtain that 
\begin{align*}
1 = X_{i-1}\cdot X_i & = \varphi (X_{i-1})\cdot \varphi (X_i) \cr
& \equiv  \varphi (X_{i-1})\cdot \varphi (X_j) \pmod 2  \cr
& = X_{i-1}\cdot X_j = 0, 
\end{align*}
another contradiction, proving that $\{i_1, i_2\} \cap \{j_1,j_2\}=\emptyset$ whenever $i\ne j$. Therefore, without loss of generality, we obtain  
\begin{equation}  \label{varphiForEvenIndicies}
	\varphi(X_i) = Y_{i-1}+Y_i, \qquad \qquad i\in \{2,4,6,\dots, ,2n\}. 
\end{equation} 

Next we turn to the determination of $\varphi (X_i)$ for $i$ odd, starting with $i=1$. Let $\varphi(X_1) = \sum _{j=1}^{2n+1} a_i Y_i$ for some integers $a_1, \dots, a_{2n+1}$. Since $X_1\cdot X_1=-3$, exactly 3 of the integers $a_1, \dots, a_{2n+1}$ are non-zero and must equal 1 or -1. Since $X_1\cdot X_i=0$ for all $i\ge 3$, and $X_1\cdot X_2=1$, given \eqref{varphiForEvenIndicies} we find
$$a_{i-1}+a_i=\left\{
\begin{array}{rl}
	0 & \quad ; \quad i\in \{4,6,\dots, 2n\}, \cr
	-1 & \quad ; \quad i=2. \cr
\end{array}
\right. $$
From this it follows that either $a_1=-1$ and $a_2=0$, or that $a_1=0$ and $a_2=-1$. It also follows that there exist exactly one additional index $i\in \{4, 6, \dots, 2n\}$ with $a_{i-1}$ and $a_i$ non-zero. Without loss of generality, we can assume that this special index is $i=4$, leading to $a_3+a_4=0$. This last equation leads again to two possibilities, namley $a_3=1$ and $a_4=-1$, or $a_3=-1$ and $a_4=1$. Combining these two options with the two choices for $a_1$ and $a_2$, leads to these four possibilities:
\vskip1mm
\begin{itemize}
	\item[(a)] $a_1 = -1$, $a_2=0$, $a_3=1$, $a_4 = -1$ and so $\varphi (X_1) = -Y_1+Y_3-Y_4$.
	\item[(b)] $a_1 = -1$, $a_2=0$, $a_3=-1$, $a_4 = 1$ and so $\varphi (X_1) = -Y_1-Y_3+Y_4$. 
	\item[(c)] $a_1 = 0$, $a_2=-1$, $a_3=1$, $a_4 = -1$ and so $\varphi (X_1) = -Y_2+Y_3-Y_4$.
	\item[(d)] $a_1 = 0$, $a_2=-1$, $a_3=-1$, $a_4 = 1$ and so $\varphi (X_1) = -Y_2-Y_3+Y_4$.
\end{itemize}
\vskip1mm
It is easy to see that all four cases are related by an isomorphism of $\mathbb Z^{2n+1}$. Cases (a) and (b) and likewise cases (c) and (d) are obtained one from the other by post-composing $\varphi$ with the isomorphism that sends $Y_i$ to $-Y_i$ for $i=3,4$ and $Y_j$ to $Y_j$ for $j\ne 3,4$. Similarly, cases (a) and (c) and likewise cases (b) and (d) are interchanged by interchanging $Y_1$ and $Y_2$ (and leaving $Y_i$ unchanged for $i\ge 3$). Therefore, it suffices to only consider one of the cases, say case (a), leading to 
\begin{equation} \label{varphiForX1}
	\varphi (X_1) = -Y_1+Y_3-Y_4.
\end{equation}
Next we turn to $\varphi(X_3)$ and write $\varphi(X_3) = \sum _{i=1}^{2n+1}b_i Y_i$ for integers $b_1, \dots, b_{2n+1}$ to be determined. For the same reason as with $\varphi(X_1)$, only three of $b_1, \dots, b_{2n+1}$ may be non-zero and can only equal 1 or -1. Reading off the intersections $X_3\cdot X_i$ for $i\ne 3$ from the Goertiz matrix $G_n$, leads to 
$$b_1-b_3+b_4=0 \quad \text{ and } \quad   b_{i-1}+b_i=\left\{
\begin{array}{rl}
	0 & \quad ; \quad i\in \{2,6,8, \dots, 2n\}, \cr
	-1 & \quad ; \quad i=4. \cr
\end{array}
\right. $$
From these we obtain that $b_3+b_4=-1$ leading to two possibilities, namely $b_3=-1$ and $b_4=0$, or $b_3=0$ and $b_4=-1$. The former case leads to $b_1=-1$ and thus also to $b_2=1$, while the latter leads to $b_1=1$ and $b_2=-1$. In summary we obtain the two cases:
\begin{itemize}
	\item[(A)] $b_1=-1$, $b_2=1$, $b_3=-1$, $b_4=0$ and so $\varphi(X_3) = -Y_1+Y_2-Y_3$.
	\item[(B)] $b_1=1$, $b_2=-1$, $b_3=0$, $b_4=-1$ and so $\varphi(X_3) = Y_1-Y_2-Y_4$.
\end{itemize}
Combining these two possibilities for $\varphi(X_3)$ with the unique choices of $\varphi(X_i)$ for $i=1,2, 4$, we obtain these two cases:
\vskip2mm
\begin{center}
\begin{tabular}{ll}
{\bf Case 1} \phantom{mmmmmmmmmmmmmmmm}  & {\bf Case 2} \cr 
$\varphi(X_1) = -Y_1+Y_3-Y_4 $  & $\varphi(X_1) = -Y_1+Y_3-Y_4 $ \cr  	
$\varphi(X_2) = Y_1+Y_2$ & $\varphi(X_2) = Y_1+Y_2$\cr
$\varphi(X_3) = -Y_1+Y_2-Y_3$ & $\varphi(X_3) = Y_1-Y_2-Y_4$\cr
$\varphi(X_4) = Y_3+Y_4$ & $\varphi(X_4) = Y_3+Y_4$\cr
\end{tabular}
\end{center}
\vskip2mm
It is natural to ask whether there exists an isomorphism of $\mathbb Z^{2n+1}$ such that post-compsing one embedding by it gives the other embedding. This is not the case, something that is not hard to show and which we leave as an exercise.

Let $\psi_i:\mathbb Z^4\to \mathbb Z^4$ for $i=1, 2$ be the embeddings from Cases 1 and 2 above. We have shows that any embedding of lattices $\varphi :(\mathbb Z^{2n}, G_n) \to (\mathbb Z^{2n+1}, -\text{Id})$ has the property that $\varphi|_{Span\{X_1, X_2, X_3, X_4\}} = \psi_i$ for i=1 or 2. 

Viewing the image of $\psi_i$ as a subset of $\mathbb Z^{2n+1}$ by identifying the domain of $\psi_i$ with the first 4 summands of $\mathbb Z^{2n+1}$, we next calculate the orthogonal complement Im$(\psi_i)^\perp$ in $\mathbb Z^{2n+1}$. Consider an element $Y=\sum _{i=1}^{2n+1} c_i Y_i \in \mathbb Z^{2n+1}$ with $c_1, \dots, c_{2n+1}\in \mathbb Z$ and such that $Y\perp \text{Im}(\psi_i)$. Then we find that: 
$$c_1-c_3+c_4=0, \quad c_1+c_2=0, \quad c_3+c_4=0, \quad \text{ and } \quad \begin{array}{cl} c_1-c_2+c_3=0 & ; \quad \text{ if } i=1, \text{ or }  \cr 
	& \cr
-c_1+c_2+c_4=0 & ; \quad \text{ if } i=2.
\end{array}$$
Solving the two systems shows that $c_i=0$ for $i=1,\dots, 4$ for both $\psi_1$ and $\psi_2$, leading to the conclusion that $(\text{Im}(\psi_i))^\perp = Span\{X_5, X_6, \dots, X_{2n+1}\}$. In particular, as we next seek to determine $\varphi(X_5), \varphi (X_7), \dots, \varphi (X_{2n-1})$, the problem we are presented with becomes identical as the one we started with, but with $G_n$ replaced by $G_{n-2}$. Thus, for instance, the computations of $\varphi (X_5)$ and $\varphi (X_7)$ follow in close analogy those for $\varphi (X_1)$ and $\varphi (X_3)$, etc. We summarize the final outcome of this process in the next theorem. 
\begin{theorem} \label{EmbeddingInCaseOfNEven}
Let $\psi_i$ for $i=1,2$ be the embeddings of lattices 
$$\psi_i: (\mathbb Z^4=Span\{X_1, X_2, X_3 X_4\}, G_2)\to (\mathbb Z^4=Span\{Y_1, Y_2, Y_3, Y_4\}, -\text{Id})$$ 
given by %
$$
\begin{array}{rlcrl}
\psi_1(X_1) & = -Y_1+Y_3-Y_4 &\qquad \qquad   & \psi_2 (X_1) & = -Y_1+Y_3-Y_4  \cr  	
\psi_1(X_2) & = Y_1+Y_2 &  & \psi_2(X_2) & = Y_1+Y_2\cr
\psi_1(X_3) & = -Y_1+Y_2-Y_3 &  & \psi_2(X_3) &= Y_1-Y_2-Y_4\cr
\psi_1(X_4) & = Y_3+Y_4 &  & \psi_2(X_4) &= Y_3+Y_4
\end{array}
$$
Let $n\ge 2$ be even, then the only embeddings of lattices $(\mathbb Z^{2n}, G_n) \to (\mathbb Z^{2n+1}, -\text{Id})$ (up to post-composition by an isomorphism of $\mathbb Z^{2n+1}$) are given by 
$$\varphi = \psi_{i_1} \oplus \dots \oplus \psi_{i_{n/2}},$$
for some choice of indicies $i_1, \dots, i_{n/2} \in \{1,2\}$. 
\end{theorem}
\begin{remark} \label{RemarkAboutFactoringOfEmbedding}
Observe that the two embeddings $\psi_1, \psi_2: (\mathbb Z^4, G_2) \to (\mathbb Z^4, -\text{Id})$ from the preceding theorem agree on the sublattice $Span\{X_1, X_2, X_4\}$. Observe also that each embedding $\varphi:\mathbb Z^{2n} \to \mathbb Z^{2n+1} \cong \mathbb Z^{2n}\oplus \mathbb Z$ is an embedding into the first $2n$ $\mathbb Z$-summands of $\mathbb Z^{2n+1}$. 
\end{remark}
\subsection{Embeddings in the case of $n$ odd}
In this section we seek to fully determine embeddings of lattices $\varphi :(\mathbb Z^{2n}, G_n) \to (\mathbb Z^{2n+2}, -\text{Id})$. The argument from the previous section can be used to obtain  \eqref{varphiForEvenIndicies}:
$$
\begin{array}{rlcl}
	\varphi(X_i) & = Y_{i-1}+Y_i, &\qquad &i = 2,4,6,\dots, ,2n,\cr
	\varphi\big|_{Span\{X_{4k+1}, X_{4k+2}, X_{4k+3}, X_{4k+4}\}} & = \psi _1 \text{ or } \psi _2, & &k=0, \dots,(n-3)/2.  
\end{array}
$$	
where $\psi_1$ and $\psi_2$ are as in Theorem \ref{EmbeddingInCaseOfNEven}. It only remains to determine $\varphi(X_{2n-1})$. Since $\varphi( X_{2n-1})\cdot \varphi(X_i)=0$ for all $i=1, \dots, 2n-2$ then 
$$\varphi (X_{2n-1}) \in \left( \text{Im}(\varphi\big|_{Span\{X_{1}, \dots,  X_{2n-2}\}} ) \right) ^\perp = Span\{Y_{2n-1}, \dotsm Y_{2n+2}\}.$$ 
Thus write $\varphi(X_{2n-1}) = aY_{2n-1}+bY_{2n}+cY_{2n+1}+dY_{2n+2}$ for $a, b, c, d\in \mathbb Z$. Only 3 of $a, b, c, d$ are non-zero, and the non-zero coefficients must equal 1 or -1. From $X_{2n-1}\cdot X_{2n} = 1$ and $\varphi (X_{2n}) = Y_{2n-1}+Y_{2n}$ we infer that $a+b=-1$. Thus either $a=-1$ and $b=0$, or $a=0$ and $b=-1$. In either case, $c$ and $d$ can be chosen at will from $\{\pm 1\}$. The two cases  
%
 \vskip2mm
 \begin{center}
 	\begin{tabular}{ll}
 		{\bf Case 1} \phantom{mmmmmmmmmmmmmmmmmm}  & {\bf Case 2} \cr 
 		$\varphi(X_{2n-1}) = -Y_{2n-1}+cY_{2n+1}+dY_{2n+2}$  & $\varphi(X_{2n-1}) = -Y_{2n}+c'Y_{2n+1}+d'Y_{2n+2}$ \cr  	
 		$\varphi(X_{2n}) = Y_{2n-1}+Y_{2n}$ & $\varphi(X_{2n}) = Y_{2n-1}+Y_{2n}$
 	\end{tabular}
 \end{center}
 \vskip2mm
with $c, c', d, d'\in \{\pm 1\}$, are easily seen to be the same up to post-composing $\varphi$ with an isomorphism of $\mathbb Z^{2n+2}$. In summary we have proved this theorem:
\begin{theorem} \label{EmbeddingInCaseOfNOdd}
Let $\psi_i$ for $i=1,2$ be the embeddings of lattices as in Theorem \ref{EmbeddingInCaseOfNEven} and let 
$$\psi_3: (\mathbb Z^2=Span\{X_{2n-1}, X_{2n}\}, G_1)\to (\mathbb Z^4=Span\{Y_{2n-1},\dots, Y_{2n+2}\}, -\text{Id})$$ 
be the embedding of lattices given by %
$$\psi_3(X_{2n-1})  = -Y_{2n-1}+Y_{2n+1}+Y_{2n+2} \qquad\text{ and }  \qquad   \psi_3 (X_{2n}) = Y_{2n-1}+Y_{2n}. $$
Then the only embeddings of lattices $\varphi : (\mathbb Z^{2n}, G_n) \to (\mathbb Z^{2n+2}, -\text{Id})$ (up to post-composition by an isomorphism of $\mathbb Z^{2n+2}$)  are given by 
$$\varphi = \psi_{i_1} \oplus \dots \oplus \psi_{i_{(n-1)/2}} \oplus \psi_3,$$
for some choice of indicies $i_1, \dots, i_{(n-1)/2} \in \{1,2\}$.  
\end{theorem}
\subsection{Equivariant Embeddings}
With the lattice embeddings 
$$\begin{array}{cl}
\varphi:(\mathbb Z^{2n}, G_n) \to (\mathbb Z^{2n+1}, -\text{Id}) &\qquad  (\text{if } n \text{ is even}), \cr
\varphi:(\mathbb Z^{2n}, G_n) \to (\mathbb Z^{2n+2}, -\text{Id}) & \qquad  (\text{if } n \text{ is odd}),
\end{array}
$$
fully enumerated in Theorems \ref{EmbeddingInCaseOfNEven} and \ref{EmbeddingInCaseOfNOdd} respectively, we now turn to the question of whether any of these embedding may be equivariant. Specifically, let $f:\mathbb Z^{2n} \to \mathbb Z^{2n}$ be the order $n$ isomorphism induced by the orientation preserving homeomorphism of  the same name $f:S^3\to S^3$ that facilitates the $n$-periodicity of $K_n$, see Figure \ref{CheckerboardDiagramForFourfoldConnectedSumOfFigureEightKnot}. It is easy to read off from said figure that 
\begin{equation} \label{DefinitionOff}
f(X_i) = \left\{
\begin{array}{cl}
X_{i+2} & \quad ; \quad i\le 2n-2, \cr
X_1 & \quad ; \quad i= 2n-1, \cr
X_2 & \quad ; \quad i= 2n.
\end{array}
\right.
\end{equation}
We ask then whether there exists an isomorphism $\tilde F:\mathbb Z^{2n+1} \to \mathbb Z^{2n+1}$ if $n$ is even, or an isomorphism  $\tilde F:\mathbb Z^{2n+2} \to \mathbb Z^{2n+2}$ if $n$ is odd, such that $\varphi\circ f = \tilde F\circ \varphi$? If such an $\tilde F$ exists, we shall call $\varphi$ an {\em equivariant embedding of lattices}. 
\begin{theorem} \label{TheoremOnNonequivarianceOfEmbeddings}
None of the lattice embeddings from Theorems \ref{EmbeddingInCaseOfNEven} and \ref{EmbeddingInCaseOfNOdd} are equivariant. 
\end{theorem}
\begin{proof} Consider first the case of $n$ even. Since the maps $\varphi, f, \tilde F$ are maps between free Abelian groups, we shall think of each as being represented by a matrix with entries $\varphi=[\varphi_{i,j}]$, $f=[f_{k,\ell}]$ and $\tilde F=[\tilde F_{r,s}]$ with $i, r, s =1, \dots, 2n+1$ and with $j,k,\ell =1, \dots, 2n$. 
	
Let $\varphi = \psi_{i_1}\oplus \dots \oplus \psi_{i_{n/2}}$ as in Theorem \ref{EmbeddingInCaseOfNEven} and note that (cf. Remark \ref{RemarkAboutFactoringOfEmbedding}) we can view $\varphi$ as a map into $\mathbb Z^{2n}$, the first $2n$ $\mathbb Z$-summands of the $\mathbb Z^{2n+1}$. Let us call this map $\tilde \varphi:\mathbb Z^{2n} \to \mathbb Z^{2n} = \mathbb Z^{2n}\oplus 0 \subset \mathbb Z^{2n}\oplus \mathbb Z \cong \mathbb Z^{2n+1}$. It is easy to see that $\det \tilde \varphi = \pm 5^{n/2}$, this follows from the readily verified facts that $\det \psi_1 = 5$ and $\det \psi_2=-5$. Thus $\tilde \varphi^{-1}$ is a well defined $(2n)\times (2n)$ matrix with rational entries. Let $\bar \varphi$ be the $(2n)\times (2n+1)$ matrix (also with rational entries) obtained from $\tilde \varphi ^{-1}$ by letting its $(2n+1)$-st column be zero. Note that then 
$$
 \varphi \cdot \bar \varphi = \left[
\begin{array}{ccc|c}
	& & & 0 \cr
	& \text{Id}_{(2n)\times (2n)} & & \vdots \cr
	& & & 0 \cr \hline 
	0 & \dots & 0 & 0 
\end{array}
\right].
$$ 

Suppose then that the matrix equation $\varphi \cdot f = \tilde F\cdot \varphi$ held true for some isomorphsim $\tilde F:\mathbb Z^{2n+1} \to \mathbb Z^{2n+1}$.  Multiply this equation from the right by $\bar \varphi$ to get $\varphi \cdot f \cdot \bar \varphi = \tilde F\cdot \varphi \cdot \bar \varphi$, which in turn implies that $\tilde \varphi \cdot f \cdot \tilde \varphi ^{-1} = [\tilde F_{r,s}]_{r,s=1,\dots, 2n}$. In particular, it follows that 
$$\tilde F_{1,1} = \sum _{k,\ell =1}^{2n} \tilde \varphi_{1,k} \cdot f_{k,\ell} \cdot \tilde \varphi ^{-1}_{\ell,1}. $$
From \eqref{DefinitionOff} it follows that 
\begin{equation} \label{FormulaForLittleF}
f_{k,\ell} = \left\{
\begin{array}{cl}
1 & \quad ; \quad (k,\ell) \in \{(3,1), (4,2), \dots, (2n, 2n-2), (1, 2n-1), (2, 2n) \}, \cr
0 & \quad ; \quad \text{otherwise}.	
\end{array}\right.
\end{equation}
leading to 
$$\tilde F_{1,1} = \sum _{\ell =1}^{2n-2} \tilde \varphi_{1,\ell+2} \cdot  \tilde \varphi ^{-1}_{\ell,1}  +  \tilde \varphi _{1,1}\cdot \tilde \varphi ^{-1}_{2n-1,1} + \tilde \varphi _{1,2}\cdot \tilde \varphi ^{-1}_{2n,1}. $$
However, regardless of the particular $\varphi$ we are working with, since it is a block-diagonal sum of $(n/2)$ matrices of size $4\times 4$, it follows that $\tilde \varphi _{1,b} =0= \tilde \varphi ^{-1}_{b,1}$ if $b\ge 5$. This further simplifies the above formula for $F_{1,1}$ to 
$$\tilde F_{1,1} =\left\{
\begin{array}{cl}
 \tilde \varphi _{1,3} \cdot \tilde \varphi ^{-1}_{1,1} + \tilde \varphi _{1,4} \cdot \tilde \varphi ^{-1}_{2,1}  & \quad ; \quad n\ge 4, \cr
  \tilde \varphi _{1,3} \cdot \tilde \varphi ^{-1}_{1,1} + \tilde \varphi _{1,4} \cdot \tilde \varphi ^{-1}_{2,1}  +   \tilde \varphi _{1,1}\cdot \tilde \varphi ^{-1}_{3,1} + \tilde \varphi _{1,2}\cdot \tilde \varphi ^{-1}_{4,1} & \quad ; \quad n=2. 
 
\end{array}
\right.$$
An explicit computation shows that 
$$\tilde F_{1,1} =\left\{
\begin{array}{rl}
 1/5 & \quad ; \quad i_1=1 \text{ and } n\ge 4, \cr
 -1/5 & \quad ; \quad i_1=2 \text{ and } n\ge 4, \cr
 2/5 & \quad ; \quad i_1=1 \text{ and } n=2, \cr
-2/5 & \quad ; \quad i_1=2 \text{ and } n=2. 
\end{array}
\right.	 
$$
Thus $\tilde F_{1,1}\notin \mathbb Z$ and so $\tilde F$ as an isomorphism on $\mathbb Z^{2n+1}$ does not exist, showing that $\varphi$ cannot be equivariant. 
\vskip3mm

We next turn to the case of $n\ge 3$ odd. According to Theorem \ref{EmbeddingInCaseOfNOdd}, $\varphi$ has the form $\varphi = \psi_{i_1}\oplus \dots \oplus \psi _{i_{(n-1)/2}}\oplus \psi _3$. As before, we view all maps as represented by matrices. 

Similarly to the first half of the proof, let $\tilde \varphi = \psi_{i_1}\oplus \dots \oplus \psi _{i_{(n-1)/2}} : \mathbb Z^{2n-2} \to \mathbb Z^{2n-2}$ so that $\varphi = \tilde \varphi \oplus \psi_3$. Thus $\tilde \varphi$ is (represented by) a $(2n-2)\times (2n-2)$ matrix while $\psi_3$ is a $4\times 2$ matrix.  It is easy to see that $\det \tilde \varphi \ne 0$, and we let $\tilde \varphi ^{-1}$ be its inverse matrix with rational entries. Let $\bar \varphi$ be the $(2n)\times (2n+2)$ matrix obtained from the $(2n-2)\times (2n-2)$ matrix $\tilde \varphi ^{-1}$ by adding 2 rows of zeros to the bottom and 4 columns of zeros to the right. Said differently, $\bar \varphi = \tilde \varphi ^{-1} \oplus \mathbb O_{2\times 4}$ where $\mathbb O_{r\times s}$ refers to the zero  matrix of type $r\times s$ . Then   
$$
\varphi \cdot \bar \varphi = \left[
\begin{array}{ccc|c}
& & &    \cr
& \text{Id}_{(2n-2)\times (2n-2)} & &  \mathbb O_{2n-2, 4} \cr  
& & & \cr \hline 
& \mathbb O_{4,2n-2} &  &  \mathbb O_{4,4}  
\end{array}
\right].
$$ 

Suppose again that an isomorphism $\tilde F:\mathbb Z^{2n+2} \to \mathbb Z^{2n+2}$ exists such that $\varphi \cdot f = \tilde F\cdot \varphi$. Then $\varphi \cdot f\cdot \bar \varphi  = \tilde F\cdot \varphi \cdot \bar \varphi = [\tilde F_{r,s}]_{r,s=1,\dots, 2n-2}$ and we shall again use this latter relation to compute $\tilde F_{1,1}$. From $\tilde F_{1,1} = \sum _{k,\ell = 1}^{2n} \varphi_{1,k}\cdot f_{k,\ell} \cdot \bar \varphi_{\ell, 1}$ and \eqref{FormulaForLittleF} it follows that 
\begin{align*}
\tilde F_{1,1} & = \sum _{\ell =1}^{2n-2} \varphi_{1,\ell+2} \cdot  \bar \varphi _{\ell,1}  +  \varphi _{1,1}\cdot \bar \varphi _{2n-1,1} + \varphi _{1,2}\cdot \bar \varphi _{2n,1} \cr
& =  \sum _{\ell =1}^{2n-4} \tilde \varphi_{1,\ell+2} \cdot  \tilde \varphi^{-1} _{\ell,1}  +  \varphi_{1,2n-1} \cdot  \bar \varphi _{2n-3,1} + \varphi_{1,2n} \cdot  \bar \varphi _{2n-2,1} +  \varphi _{1,1}\cdot \bar \varphi _{2n-1,1} + \varphi _{1,2}\cdot \bar \varphi _{2n,1} 
\end{align*}
Since $n\ge 3$ then $\varphi_{1,2n-1}, \varphi_{1,2n}, \bar \varphi_{2n-1,1}, \bar \varphi _{2n,1} $ equal zero, and likewise $\tilde \varphi _{1,b} = 0 = \tilde \varphi^{-1}_{b,1}$ when $b\ge 5$. These observations turn the above formula into 
$$\tilde F_{1,1} = \tilde \varphi_{1,3} \cdot  \tilde \varphi^{-1} _{1,1} + \tilde \varphi_{1,4} \cdot  \tilde \varphi^{-1} _{2,1} = \left\{
\begin{array}{rl}
1/5 & \quad ; \quad i_1=1, \cr
-1/5 & \quad ; \quad i_1=2. 	
\end{array}
\right.  
$$
Thus $\tilde F_{1,1}\notin \mathbb Z$ and so the isomorphism $\tilde F:\mathbb Z^{2n+2}\to \mathbb Z^{2n+2}$ again cannot exist. 	
\end{proof}
\vskip3mm
Lastly, combining Theorems \ref{EmbeddingInCaseOfNEven}, \ref{EmbeddingInCaseOfNOdd} and \ref{TheoremOnNonequivarianceOfEmbeddings} with Corollaries \ref{CorollaryEquivariantDiagramForMobiusBand} and \ref{CorollaryEquivariantDiagramForKleinBottle} (see also Section \ref{SummarySection}) and Remark \ref{RemarkAboutPositiveDefiniteMatrixForKn}, proves Theorem \ref{main}. 
\bibliographystyle{plain}
\bibliography{bibliography}

\end{document}